\documentclass[11pt]{amsart}
\usepackage{mathabx,amsmath,amsfonts,amsthm,mathrsfs,amssymb,amscd,comment,enumerate,amsxtra,url,tikz-cd,physics,tcolorbox,graphicx, setspace, mathtools,stmaryrd}
\input{xypic}
\xyoption{all}
\usepackage{xcolor} 
\colorlet{mdtRed}{red!50!black}
\definecolor{dblue}{rgb}{0,0,.6}
\usepackage[colorlinks]{hyperref}
\usepackage[normalem]{ulem}
\usepackage{cancel,bm}
\hypersetup{linkcolor=blue,citecolor=dblue,filecolor=dullmagenta,urlcolor=mdtRed}

\newtcolorbox{mymathbox}[1][]{colback=white, sharp corners, #1}

\newtheorem{theorem}{Theorem}
\newtheorem{corollary}[theorem]{Corollary}
\newtheorem{lemma}[theorem]{Lemma}
\newtheorem{proposition}[theorem]{Proposition}
\newtheorem{definition}[theorem]{Definition}

\newtheorem*{theorem*}{Theorem}
\newtheorem*{corollary*}{Corollary}

\theoremstyle{remark}
\newtheorem{remark}[theorem]{\bf Remark}
\newtheorem{example}[theorem]{\bf Example}

\newcommand{\e}{\epsilon}
\newcommand{\Z}{\mathbb{Z}}
\newcommand{\C}{\mathbb{C}}
\newcommand{\Adash}{\acute{\mathcal{A}}}

\newcommand{\Bdash}{\acute{\mathcal{B}}}
\newcommand{\Q}{\mathcal{Q}}

\renewcommand{\P}{\mathcal{P}}
\renewcommand{\O}{\mathcal{O}}
\newcommand{\ch}{{\rm ch}}
\newcommand{\p}{\partial}

\newcommand{\mf}[1]{\mathfrak{#1}}
\newcommand{\mF}[0]{\mathcal F}

\newcommand{\mb}[1]{\mathbb{#1}}
\newcommand{\mc}[1]{\mathcal{#1}}

\newcommand{\F}{\mathbb{F}}
\newcommand{\D}{\mathbb{D}}
\newcommand{\E}{\mathbb{E}}

\newcommand{\M}{\mathcal{M}}
\newcommand{\Mdash}{\acute{\mathcal{M}}}
\newcommand{\Pidash}{\acute{\Pi}}

\newcommand{\spec}{\text{Spec}}

\newcommand{\U}{\mathcal{U}}
\newcommand{\K}{\mathcal{K}}

\newcommand{\W}{\mathcal{W}}

\newcommand{\G}{\mathcal{G}}
\newcommand{\V}{\mathcal{V}}

\newcommand{\coker}{\text{coker}}
\newcommand{\im}{\textnormal{im}\,}

\newcommand{\rk}{\textnormal{rank}}
\newcommand{\id}{\textup{id}}
\newcommand{\wt}[1]{\widetilde{#1}}
\newcommand{\wh}[1]{\widehat{#1}}

\newcommand{\EXT}[0]{\mathscr{E}xt}

\newcommand{\HOM}[0]{\mathscr{H}om}

\newcommand{\vp}[0]{\varphi}
\newcommand{\quot}[0]{\textnormal{Quot}}

\newcommand{\vir}[0]{\textnormal{vir}}
\newcommand{\inv}[0]{\textnormal{inv}}
\newcommand{\td}[0]{\textnormal{td}}
\newcommand{\End}[0]{\textnormal{End}}
\newcommand{\pa}[0]{\textnormal{pa}}

\renewcommand{\L}[0]{\mathcal L}
\newcommand{\pl}[0]{\textnormal{pl}}
\newcommand{\pe}[0]{\textnormal{pe}}
\newcommand{\chh}[0]{\textnormal{ch}^\textnormal{H}}
\newcommand{\chhE}[0]{\textnormal{ch}^{\textnormal{H},\E}}
\newcommand{\chhF}[0]{\textnormal{ch}^{\textnormal{H},\F}}
\newcommand{\SSym}{\text{SSym}}
\newcommand{\sym}{\textnormal{sym}}

\newcommand{\CH}[0]{\textnormal{CH}}

\newcommand{\ad}[0]{\textnormal{Ad}}
\newcommand{\sfM}{\mathsf M}
\newcommand{\sfMdash}{\acute{\mathsf M}}
\newcommand{\perf}[0]{\textbf{Perf}}

\newcommand{\fund}{\textnormal{fund}}
\newcommand{\PS}{\textnormal{PS}}
\newcommand{\Pic}{\textnormal{Pic}}

\renewcommand{\ss}[0]{\text{ss}}
\newcommand{\st}[0]{\text{st}}

\newcommand{\res}{\textnormal{res}}
\newcommand{\ex}{\textnormal{e}}
\newcommand{\kv}{{\mathcal K^\vee}}
\newcommand{\mL}{\mathsf{L}}
\newcommand{\du}{d^\vee}

\numberwithin{theorem}{section}

\begin{document}
	
	\title[Joyce's invariant and Virasoro Constraints for Quot schemes]{Joyce's invariant and Virasoro Constraints for \\ Quot schemes on curves}
    
    \author[Parvez Rasul]{Parvez Rasul} 
	
	\address{School of Mathematics, Korea Institute for Advanced Study, 85 Hoegiro, Dongdaemun-gu, Seoul 02455, Republic of Korea}

	\email{\href{mailto:rasulparvez@kias.re.kr}{rasulparvez@kias.re.kr}}

    \begin{abstract}
    Let $C$ be a smooth projective curve over $\mb C$ and let $E$ be a vector bundle over $C$.
    Let $\quot_{r,d}(E)$ denote the Quot scheme which parametrizes quotients of $E$ of rank $r$ and degree $d$.
    Following Joyce's recipe \cite{Joy21}, we introduce Joyce's enumerative invariant for the Quot scheme $\quot_{r,d}(E)$.
    The invariant can be viewed as a generalization of the virtual fundamental cycle of the Quot scheme. 
    We evaluate intersection pairings on the Quot scheme $\quot_{\rk(E)-1,d}(E)$ by computing its invariant explicitly.
    Following the reformulation of sheaf-theoretic Virasoro constraints
    in terms of Joyce's vertex algebra framework in \cite{BLM24},
    we give a proof of the Virasoro constraints for the Quot scheme $\quot_{r,d}(E)$.
    With the help of these constraints, 
    we compute the (virtual) intersection numbers of $f$-classes on the Quot schemes.
    \end{abstract}
    
	\maketitle
    
    \section{Introduction}
    Let $C$ be a smooth projective curve over $\mb C$ and let $E$ be a vector bundle over $C$.
    Let $r$ and $d$ be integers such that $0 < r < \rk(E)$.
    Let $\quot_{r,d}(E)$ denote the Quot scheme which parametrizes quotients of $E$ of rank $r$ and degree $d$.
    Quot schemes have played a central role in the study of vector bundles on curves, particularly in constructing and understanding the geometry of the moduli spaces.
    In general, Quot schemes are not smooth and may not even be irreducible.
    When $d \gg 0$, it is proved in \cite{PR03} that $\quot_{r,d}(E)$
    is irreducible and in \cite{GS24}
    that $\quot_{r,d}(E)$ is a local complete intersection.
    Quot schemes also arise as a compactification of the space of maps into Grassmannians.
    Thus, Quot schemes appear naturally in enumerative geometry. 
    One of the main results in this direction, due to \cite{Ber94}, shows that when the Quot scheme is irreducible, intersection pairings on the Quot scheme compute the number of certain maps from $C$ to the Grassmannian.
    However, Quot schemes are not irreducible in general and may have components of dimension larger than expected. 
    To make sense of intersection theory in such cases, 
    \cite{MO07c} constructed a virtual fundamental class for the Quot scheme 
    $$[\quot_{r,d}(E)]^\vir \in H_\bullet(\quot_{r,d}(E))$$
    and computed various virtual intersection numbers on the Quot scheme.
    Intersection theory on Quot schemes has been further developed and applied in various contexts, for example, in \cite{MO10}, \cite{MO07a}, \cite{MO07b}, and \cite{Mar25}.
    Beyond curves, Quot schemes on higher-dimensional varieties play an important role in Donaldson–Thomas theory and in the study of enumerative invariants of surfaces and threefolds. Some recent works in these directions include \cite{OP21}, \cite{AJLOP}, \cite{Boj25}, \cite{Ric20}, \cite{SR21}, and \cite{FMR21}.
    In \cite{Boj25}, the author studies the virtual fundamental classes and the tautological integrals of the punctual Quot schemes on surfaces and Calabi-Yau fourfolds,
    using the wall-crossing framework of Joyce (\cite{Joy21}).
    Joyce's framework provides a new approach to tackling computations with the virtual fundamental classes.
    In \cite{Bu23}, the author computes Joyce's invariant for
    the moduli space of semistable sheaves and computes cohomology pairings using these invariants.
    
    In this article, we study intersection theory on the Quot scheme
    $\quot_{r,d}(E)$ through the framework of Joyce \cite{Joy21}.
    We introduce an enumerative invariant for the Quot scheme $\text{Quot}_{r,d}(E)$, which we will refer to as Joyce's invariant.
    This invariant sits in the homology of a larger space
    and coincides with the pushforward of the virtual fundamental cycle $[\quot_{r,d}(E)]^\vir$.
    Further, we prove a wall-crossing formula for this invariant,
    which gives a different method to compute the virtual fundamental cycle and integrals over it.
    We compute Joyce's invariant for the Quot scheme of rank 1 subsheaves and use it to compute intersections of tautological classes on the Quot scheme $\quot_{\rk(E)-1,d}(E)$.
    
    Another interesting use of Joyce's wall-crossing formula is the establishment of Virasoro constraints, due to \cite{BLM24}.
    Given a moduli space $M$
    with a virtual fundamental class $[M]^\vir$,
    one can obtain numerical invariants, called descendents, by integrating tautological cohomology classes against the virtual fundamental class $[M]^\vir$.
    The Virasoro constraints are some universal relations among these numerical invariants.
    The Virasoro constraints first appeared in the study of 
    Gromov-Witten theory and intersection theory on the moduli space of stable curves, in Witten’s foundational paper \cite{Wit90}.
    In \cite{Wit90}, Witten conjectured that the integrals of descendents, which are cohomology classes on the moduli space of stable curves $\overline{\mc M}_{g,n}$, satisfy some explicit relations. 
    Witten's conjecture was proved in the famous work of Kontsevich \cite{Kon92} and later studied through different approaches e.g. in \cite{OP09} and \cite{Mir07}.
    In \cite{EHX97}, the authors extended the Virasoro conjecture to $\overline{\mc M}_{g,n}(X)$ for certain varieties $X$.
    We refer the reader to \cite{BLM24}, \cite{Boj24} and the references therein for further details. 
    These developments indicate that Virasoro constraints form a rich theory themselves and reflect deeper structures in enumerative geometry.
    The sheaf-theoretic Virasoro constraints were first formulated and proved in \cite{BLM24} for any smooth projective curve and certain smooth projective surfaces.
    In \cite{BLM24}, the authors considered the vertex algebra structure on the homology of a large stack containing all moduli schemes of sheaves, constructed by Joyce in \cite{Joy19},
    and proved that the Virasoro constraints are compatible with Joyce's wall-crossing.
    One of the most remarkable aspects of this approach is its apparent generality.
    One may formulate the Virasoro constraints in other contexts if Joyce’s vertex algebra framework applies.
    Further, the study of Virasoro constraints via wall-crossing methods has been followed in \cite{Boj24}, \cite{LM24} for moduli spaces of Quiver representations and in \cite{Mor25} for moduli space of parabolic bundles.
    In this article, we formulate and prove the Virasoro constraints for the Quot scheme $\quot_{r,d}(E)$ using Joyce's vertex algebra framework and the wall-crossing formula.
    The Virasoro constraints will give rise to a few useful relations among the integrals of the tautological cohomology classes.
    
    We briefly discuss the main results of the article. 
    Detailed definitions will be given in the subsequent sections.
    We use the vertex algebra introduced by Joyce,
    in \cite{Joy19, GJT22}.
    Given a (higher) moduli stack $\M$ parametrizing objects of the derived category $D^b(\text{Coh}(C))$,
    Joyce constructed a vertex algebra structure on its homology
    $V_\bullet := \wh H_{\bullet}(\M)$.
    Here the symbol $\wh{H}_\bullet$ denotes an appropriate shift in the grading of the homology.
    An important data in the vertex algebra is the translation operator $T: V_\bullet \to V_{\bullet+2}$,
    obtained using the
    $B\mb G_m$-action on $\M$.
    Associated to the vertex algebra $V_\bullet$, there is a 
    Lie algebra $V_\bullet/T(V_{\bullet-2})$,
    obtained as the quotient by the translation operator.
    Given a moduli space $M$ of semistable sheaves,
    \cite{Joy19} defines a class inside the Lie algebra,
    $$[M]_\inv \in V_\bullet/T(V_{\bullet-2})\,.$$
    In this article, we follow \cite{Joy19} to construct a similar
    invariant for $\quot_{r,d}(E)$.
    We consider a (higher) stack $\Mdash$ parametrizing morphisms
    $\vp: E^\vee \otimes U \to F$, where $F$ is a sheaf on $C$
    and $U$ is a vector space.
    Using methods of \cite[Chapter 8]{Joy19}, 
    we construct a vertex algebra $\acute V_\bullet$ whose underlying vector space is the (shifted) homology $\wh H_\bullet(\Mdash)$,
    with the data of the translation operator 
    $\acute{T}: \acute V_\bullet \to \acute V_{\bullet+2}$.
    The translation operator $\acute{T}$ is obtained using the
    $B\mb G_m$-action on $\Mdash$.
    Consequently, we will have a graded Lie algebra
    $\acute V_\bullet/\acute T(\acute V_{\bullet-2}) \cong \wh H_{\bullet}(\Mdash/B \mb G_m)$
    and an open embedding $\mf f: \quot_{r,d}(E) \hookrightarrow \Mdash/B \mb G_m$.
    Using Joyce's recipe, we show the following.
    \begin{theorem*}[Theorem \ref{theorem wall crossing for quot scheme}]
        Let $r$ and $d$ be integers such that
        $0<r<\rk(E)$.
        There is a unique class 
        $$[\quot_{r,d}(E)]_\inv \in \wh H_{0}(\Mdash/B \mb G_m)\,,$$
        which is equal to the pushforward 
        $$\mf f_*[\quot_{r,d}(E)]^\vir\,.$$
        The invariant satisfies the wall-crossing formula
        \begin{align*}
            [\quot_{r,d}(E)]_\inv
            =\sum_{\substack{n\geqslant 1, \sum \alpha_i = \alpha, \\ 
            \alpha_i=(r_i,d_i), r_i>0 \text{ for all }i,\\
            \mu_{\min}(E^\vee) \leqslant d_1/r_1\leqslant\cdots\leqslant d_n/r_n}} &
            \frac{(-1)^{n}}{\prod_{i=1}^l (a_i - a_{i-1})!} \cdot \\
            &  \left[\left[ \cdots
     \left[ \left[\mathsf e^{((0,0),E^\vee)},\sfM_{\alpha_1} \right], \sfM_{\alpha_2} \right],\dots \right], \sfM_{\alpha_n} \right]\,.
    \end{align*}
    where $\mathsf e^{((0,0),E^\vee)}$ is the class of the single point set $\{(0,\C,0)\}$ in $\wh H_{0}(\Mdash/B \mb G_m)$,
    $\sfM_{(r_i,d_i)}$ denotes the (Joyce's) invariant
    $[M_{(r_i,d_i)}]_\inv$
    for the moduli space of semistable sheaves of rank $r_i$ and degree $d_i$.
    The numbers $0=a_0<\cdots<a_l=n$ are defined such that for any $0<i<n$, we have $d_i/r_i<d_{r+1}/r_{i+1}$ if and only if $i=a_k$ for some $0<k<l$.
    \end{theorem*}

Using an explicit description of the algebra $H^{\bullet}(\Mdash)$,
given by \cite{gros20}, and
using the wall-crossing formula, we compute the invariant
$[\quot_{\rk(E)-1,d}(E)]_\inv$ in Theorem \ref{theorem invariant quot N-1}.
We use the results and computations from \cite{Bu23}, 
where a similar computation is done when $E$ is a line bundle.
\begin{theorem*}[Theorem \ref{theorem invariant quot N-1}]
    Let $N =\rk(E)$ and $\nu = \chi(E)$.
    For any degree $d$, define $\du = d- \deg(E)$.
    The invariant 
    $$[Q_{N-1,d}(E)]_\inv \quad  \in \quad \wh H_0(\Mdash)/\im \acute T$$ is given by
    \begin{equation*}
    [Q_{N-1,d}(E)]_\inv = 
    \res_{z=0} \left( \frac{1}{z^{\nu+N\du}} \cdot \rho(z) \cdot \sigma\left(\frac{N}{z} -s_{1,2,2}\right)\right) + \textnormal{im} \acute{T}\,,
    \end{equation*}
    where $\rho$ and $\sigma$ are defined in \eqref{equation definition rho}
    and $s_{j,k,l}$'s are tautological homology classes defined in Section \ref{section computing invariant}.
\end{theorem*}
We use this invariant to compute virtual intersections on $\quot_{\rk(E)-1,d}(E)$.
Let 
$$p_C : C \times \quot_{r,d}(E) \to C \quad \text{ and }
\quad p_Q : C \times \quot_{r,d}(E) \to \quot_{r,d}(E)$$
denote the projection maps.
There is a universal exact sequence on $C \times \quot_{r,d}(E)$,
$$0 \to \K \to p_C^*E \to \Q \to 0\,.$$
Let 
$$\{1=\e_{1,0}, \e_{1,1}, \dots,\e_{2g,1},\e_{1,2}\}$$
be a symplectic basis for the cohomology ring of $C$ such that $\e_{*,k} \in H^k(C;\mb C)$.
Consider the Kunneth decomposition of the Chern classes of the dual of the universal kernel,
$$\qquad  c_i(\K^\vee) = a_i \otimes 1 + \sum_{j=1}^{2g} b_i^j \otimes \e_{j,1} + f_i \otimes \e_{1,2} \qquad \text{ for }i=1,\dots,\rk(E)-r \,.$$
The virtual intersection of the $a,b,f$ classes are of interest in enumerative geometry.
Intersections of $a$-classes were studied extensively in \cite{Ber94}, when degree $d$ is large 
and $E = \O_C^{\oplus N}$ is a trivial bundle.
It is shown there that the intersection number of a top-degree monomial in the $a$-classes on the fundamental cycle counts the number of degree $d$ maps from
$C$ to the Grassmannian $G(r,N)$ which send 
fixed points on $C$ to special Schubert
subvarieties of $G(r,N)$. 
The number of such maps, also called the Gromov invariants for Grassmannians, is computed by the Vafa-Intriligator formula \cite{Int91}.
In \cite{MO07c}, 
the authors used virtual localization formula to compute the virtual intersections of $a,b$ and $f$-classes explicitly, using the $\C^*$-action on $E = \O_C^{\oplus N}$.
For any vector bundle $E$, the $a$-intersections still have a similar enumerative meaning and can be computed using the Gromov invariants for Grassmannians,
see \cite{Hol04}.
In this article, using the explicit description of the invariant $[\quot_{\rk(E)-1,d}(E)]_\inv$, we compute the virtual intersection numbers of $a$ and $b$ classes.
Note that there is no non-constant $f$-class in this case.
\begin{theorem*}[Theorem \ref{theorem intersections on Quot N-1}]
    Let $0 \leqslant s \leqslant g$ and $1 \leqslant j_1 < \cdots < j_s \leqslant g$ be integers such that the product $\prod_{i=1}^s b_1^{j_i}b_1^{j_i+g}$ is nonzero.
    Let $\rk(E) = N$ and 
    $\dim Q$ denote the expected dimension of the Quot scheme $\quot_{N-1,d}(E)$. Then
    \begin{equation*}
        \int\limits_{[\quot_{N-1,d}(E))]^\vir}^{} \,
        \left(\prod_{i=1}^s b_1^{j_i}b_1^{j_i+g}\right) \cdot 
        (a_1)^{\dim Q-s} = N^{g-s}\,.
    \end{equation*}
\end{theorem*}    
This provides a different approach to evaluating intersection numbers on Quot schemes via Joyce's invariant and the wall-crossing formula.
A similar explicit description of Joyce's invariant for the Quot scheme in the general case,
i.e. when $r<\rk(E)-1$, could in principle be used to compute intersection numbers on those spaces as well. 
However, determining the invariants in this generality appears to be more involved and will be taken up in future work.

We now turn to the Virasoro constraints.
The Virasoro constraints are some universal relations among the intersection numbers of $a,b$ and $f$ classes.
Following \cite{BLM24}, we formulate the Virasoro constraints using descendents, which are defined via Chern characters rather than Chern classes.
Let $\D^C$ denote the supercommutative algebra generated 
by the symbols $\chh_i(\gamma)$, for $i \geqslant 0$ and $\gamma \in H^{\bullet}(C;\mb C)$.
We call these symbols holomorphic descendents.
We can realize the descendents as cohomology classes on the Quot scheme
via the geometric realization map 
$$ \xi_{\K^\vee}: \D^C \to H^{\bullet}(\quot_{r,d}(E))$$
given by
$$ \xi_{\K^\vee}(\chh_i(\gamma)) = {p_Q}_*(\ch_{i+1-p}(\K^\vee)\,(p_C^*\gamma)) \quad \quad \text{for } \gamma \in H^{p,q}(C)\,.$$
The integrals of the descendents are enumerative invariants of a moduli space. 
The Virasoro constraints say that these numbers satisfy some explicit universal relations.
These relations are stated using certain operators 
(given in Definition \ref{definition virasoro operators}),
$$L_m : \D^C \to \D^C 
\qquad \text{ for } m \geqslant 0\,.
$$
We prove the following theorem.
\begin{theorem*}[Theorem \ref{Theorem Virasoro constraints}]
    The Quot scheme $\quot_{r,d}(E)$ satisfies the Virasoro constraints, i.e.,
    $$\int\limits_{[\quot_{r,d}(E)]^\vir} \xi_{\K^\vee} (L_m(D)) = 0 \quad \text{ for any } m\geqslant 0, D \in \D^C\,.$$
    \end{theorem*}
The proof of this theorem is along the lines of \cite{BLM24}.
The main idea of the proof is the compatibility between 
Virasoro constraints and Joyce’s wall-crossing.
In \cite{BLM24}, the authors considered the vertex algebra $V^\pa_\bullet$,
constructed from the homology of a 
large stack parametrizing pairs of sheaves,
and showed that the Virasoro operators can be derived from a natural conformal element in the vertex algebra $V_\bullet^\pa$.
As a consequence, it is proved that the Virasoro constraints are compatible with Joyce's wall-crossing on $V_\bullet$ and $V^\pa_\bullet$.
We show that the Virasoro constraints for
Quot schemes can be derived using this compatibility,
viewing the wall-crossing formula for Quot schemes in $V^\pa_\bullet$.
The Virasoro constraints produce various relations among 
the integrals of $a,b$ and $f$-classes,  
which is illustrated in the final section.
Using the relations, we conclude that the $f$-classes can be eliminated from the computation of any intersection numbers.
\begin{theorem*}[Theorem \ref{theorem intersection fl times P}]
    For any top-degree polynomial $P$ in $a,b$ and $f$-classes, 
    there exists a polynomial $\wt P$, involving only $a$ and $b$-classes,
    such that
    $$ \int\limits_{[\quot_{r,d}(E)]^\vir} P(a,b,f) = \int\limits_{[\quot_{r,d}(E)]^\vir} \wt P(a,b)\,.$$
\end{theorem*}
    This result gives a method to intersect the $f$-classes,
    provided we know the $a$ and $b$-intersection numbers.
    However, it is not easy to find an elegant general formula to compute $\wt P$ for any $P$.
    We give a recursive formula which helps to compute $\wt P$.
    \begin{theorem*}[Theorem \ref{theorem intersection fl times P}]
    Let $\dim Q$ denote the expected dimension of the Quot scheme
    $\quot_{r,d}(E)$ and $1 < l < \rk(E)-r$.
    Let $G(a,b,f)$ be any polynomial 
    of weighted degree $2(\dim Q - l)$.
    Then 
    \begin{align*}
    \nonumber
        (-1)^{l} \cdot \rk(E) \int f_{l+1} \cdot G = &\int R_l(G) + 
        (1-g) \sum_{i+j=l} \int \mu_i(a) \mu_j(a) \, G 
        \ - \chi(E) \cdot \int \mu_l(a)\, G \\
        & - \frac{\rk(E)}{l+1} \cdot \int \left(
        \sum_{i=1}^{l} \frac{\p \mu_l(a)}{\p a_i} \, f_i 
        - \frac{1}{2} \sum_{i,k=1}^{l+1} \frac{\p^2 \mu_l(a)}{\p a_i \p a_k} \, \Gamma_{ik}
        \right) \cdot G \,.
    \end{align*}
    where $R_l$ is the derivation part of the Virasoro operator (Definition \ref{definition virasoro operators}),
    $\mu_l$'s are the polynomials appearing in Newton’s identities to express power sums in terms of elementary symmetric polynomials, with the notation $\mu_l(a)=\mu_l(a_1,\dots,a_l)$,
    and $\Gamma_{ik} = \sum_{j=1}^g (b_i^j b_k^{j+g} + b_k^j b_i^{j+g})$.
    The integrals are taken over the virtual fundamental class $[\quot_{r,d}(E)]^\vir$.
    \end{theorem*}
    Furthermore, in Example \ref{example integral in dim 2}, 
    we compute some $f$-intersections explicitly in terms of $a,b$-intersections when $\dim \quot_{r,d}(E) =2$.

\vspace{.3cm}
The organization of the article is as follows.
In Section \ref{Section required objects}, we define the objects needed to apply Joyce's theory and their relation with Quot schemes.
In Section \ref{section Vertex Algebra Structure and Invariants},
we use Joyce's theory to introduce the invariant $[\quot_{r,d}(E)]_\inv$
and prove a wall-crossing formula for the invariant.
In Section \ref{section computing invariant},
we compute the invariant $[\quot_{\rk(E)-1,d}(E)]_\inv$.
In Section \ref{section virasoro constraints}, we formulate and prove the Virasoro constraints for Quot schemes.
Section \ref{section intersection theory} is used to study virtual intersections on Quot schemes using Joyce's invariant and the Virasoro constraints.

\subsection*{Acknowledgements}
I thank Young-Hoon Kiem for suggesting the problem on Virasoro constraints,
for introducing me to \cite{BLM24} and for several helpful discussions.
I thank Woonam Lim and Chenjing Bu for useful discussions.
This work is supported by KIAS individual
Grant (ID MG104001) from Korea Institute of Advanced study.

\section{Required Objects}\label{Section required objects}
    We build the necessary setup to apply results from \cite{Joy21}.
    Let $C$ be a smooth projective curve over $\C$ and $E$ be a vector bundle on the curve, which will be fixed throughout the article.
    Let $E^\vee$ denote the dual of the bundle $E$.
    \begin{definition}\label{Definition category of pairs}
    We define an $E^\vee$-pair to be a triple $(F,V,\vp)$ 
    where $F$ is a coherent sheaf on $C$, $V$ is a finite-dimensional $\C$-vector space and $\vp: E^\vee \otimes V \to F$ is a morphism of sheaves.
    A morphism $(\psi,\theta) : (F_1,V_1,\vp_1) \to (F_2,V_2,\vp_2)$
    between two $E^\vee$-pairs consists of a morphism of sheaves $\psi: F_1 \to F_2$
    and a $\C$-linear map $\theta : V_1 \to V_2$ such that 
    $\vp_2 \circ (Id_{E^\vee} \otimes \theta)  = \psi \circ \vp_1$.
    Let $\Adash_{E^\vee}$ denote the category of all such $E^\vee$-pairs.
    \end{definition}
    The composition of morphisms can be defined in the obvious way.
    We call $(F',V',\vp')$ a sub-object  of $(F,V,\vp)$ if there exists a morphism $i = (i_F,i_V) : (F',V',\vp) \to (F,V,\vp)$ such that both the maps $i_F$ and $i_V$ are injective.
    Clearly, $\Adash_{E^\vee}$ is an abelian category.
    We will omit $E^\vee$ from the notation of $\Adash_{E^\vee}$ and write only $\Adash$ as $E$ is fixed.

    As in \cite[Definition 8.4]{Joy21},
    a (higher) moduli stack $\Mdash$ parametrizing objects of $D^b(\Adash)$ can be constructed using Toën–Vaquié derived moduli stack \cite{TV07}.
    It admits universal complexes 
    $\U$ over $C \times \Mdash$, $\V$ over $\Mdash$ and a universal pair on $C \times \Mdash$,
    $$ \Theta : E^\vee \boxtimes \V \to \U \,.$$
    One can form the projective linear (or rigidified) moduli stack 
    $\Mdash^\pl: = \Mdash/B\mb G_m$, 
    the quotient of $\Mdash$ by the action of the group stack $B\mb G_m$.
    There is a projection map 
    $$\Pidash^\pl : \Mdash \to \Mdash^\pl$$
    which is a $B\mb G_m$-bundle away from the point $(0,0,0) \in \Mdash^\pl$.

    Let $K(C)$ denote the numerical Grothendieck group of $\text{Coh}(C)$.
    An element $\alpha \in K(C)$ is a pair $(r,d)$ where $r$ denotes the rank and $d$ denotes the degree of a sheaf.
    Sometimes we will write $\alpha(F)$ to mean cohomological type of $F$, i.e. $\alpha(F) = (\rk(F),\deg(F))$.
    We write $\alpha>0$ if either $r>0$ or $r=0$ with $d>0$.
    Define $$K(\Adash) := \{(\alpha,e): \alpha \in K(C) \text{ and } e \in \Z\} \,.$$
    We define the type of an $E^\vee$-pair $(F,V,\vp)$ to be the tuple $(\alpha(F),\dim V) \in K(\Adash)$.
    We call $(\alpha,e)>0$ if either $e>0$ or $e=0$ with $\alpha>0$.
    Consider the positive cone 
    $$ C(\Adash) := \{ (\alpha,e) \in K(\Adash) : (\alpha,e)>0  \}\,.$$
    The moduli stacks $\Mdash$ and $\Mdash^\pl$ can be decomposed as
    $$\Mdash = \bigsqcup_{(\alpha,e) \in C(\Adash)} \Mdash_{\alpha,e} \quad \text{ and } \quad \Mdash^\pl = \bigsqcup_{(\alpha,e) \in C(\Adash)} \Mdash_{\alpha,e}^\pl$$
    where $\Mdash_{\alpha,e}$ and $\Mdash_{\alpha,e}^\pl$ are connected components containing pairs of the type $(\alpha,e)$.
    
    \vspace{.5cm}
    We define a family $\mathscr S:= \{\mu^\delta: \delta>0\}$ of stability conditions on $\Adash$, where $\mu^\delta$ is defined as follows,
    \begin{equation}\label{stability definition}
        \mu^\delta(F,V,\vp) = 
        \begin{cases}
            \frac{\deg(F)}{\rk(F)} + \frac{\delta \cdot \dim V}{\rk(F)} & \textnormal{ if } \rk(F)>0\,,\\
            0 & \textnormal{ if } \rk(F)=0\,.
        \end{cases} 
        \end{equation}
    Note that the stability condition depends only on the type of the object.
    For any $(\alpha,e) \in C(\Adash)$, we have the open substacks of $\mu^\delta$-stable and $\mu^\delta$-semistable pairs
    $$\Mdash_{\alpha,e}^{\delta-\st} \subset \Mdash_{\alpha,e}^{\delta-\ss}
    \subset \Mdash^\pl_{\alpha,e}$$
    inside the rigidified stack.
    For any $\delta>0$ and $e=1$, the substack $\Mdash_{\alpha,1}^{\delta-\st}$ is a fine moduli space, as shown in \cite{Lin18}.
    Let $q_C: C \times \Mdash_{\alpha,1}^{\delta-\st} \to C$ 
    and $q_{\Mdash}: C \times \Mdash_{\alpha,1}^{\delta-\st} \to \Mdash_{\alpha,1}^{\delta-\st} $ denote the projections.
    There is a universal pair
    $$ \Psi: q_C^*E^\vee \to \mF\,. $$
    on $C \times \Mdash_{\alpha,1}^{\delta-\st}$.
    The moduli space $\Mdash_{\alpha,1}^{\delta-\st}$ has a natural 2-term perfect obstruction theory given by
    \begin{equation}\label{equation POT on stable pairs}
        R{(q_{\Mdash})}_* R\HOM([q_C^*E^\vee \to \mF], \mF)\,.
    \end{equation}
    This produces a virtual fundamental class $$[\Mdash_{\alpha,1}^{\delta-\st}]^\vir \in H_\bullet(\Mdash_{\alpha,1}^{\delta-\st})$$
    due to Behrend-Fantechi \cite{BF97}.
    
    We study the extreme chambers of the stability conditions
    in the case $e =1$.
    Let us fix a cohomological type $\alpha \in K(C)$ with $0<\rk(\alpha)<\rk(E)$.
    For $\delta$ sufficiently large or sufficiently small, there are no strictly semistable pairs, 
    i.e., $\Mdash^{\delta-\ss}_{\alpha,1} = \Mdash^{\delta-\st}_{\alpha,1}$.
    In the right-most chamber, i.e., when $\delta \gg 0$,
    the moduli space $\Mdash^{\infty-ss}_{\alpha,1}$ is a Quot scheme with a different parameter,
    which we explain now.
    For ease of notation, let us denote the moduli space
    $\Mdash^{\infty-ss}_{\alpha,1}$ by $Q(E,\alpha)$, i.e.,
    $$Q(E,\alpha):= \Mdash^{\infty-ss}_{\alpha,1}\,.$$
    The following proposition gives a description for the closed points of $Q(E,\alpha)$.
    \begin{proposition}\label{proposition stability and gg}
    Let $\alpha = (r,d)$ with $0<r<\rk(E)$. 
        Let $\delta$ be sufficiently large, precisely, 
        $$\delta > \delta_{\infty} := \max_{0<s<r}
        \left\{ \frac{(r-s)(d-r \cdot \mu_{\min}(E^\vee))}{s}
        \right\}\,.
        $$
        Then an $E^\vee$-pair $(F,\C,\vp)$ of type $(\alpha,1)$ is $\delta$-semistable iff 
        $F$ is pure and
        the map $\vp$ is generically surjective.
    \end{proposition}
    \begin{proof}
        Let $(F,\C,\vp)$ be a pair such that $\vp$ is generically surjective.
        Let us consider a sub-object $(F',V',\vp')$ of $(F,\C,\vp)$.
        If $\rk(F') = \rk(F)$ then clearly 
        $\mu^\delta((F',V',\vp')) < \mu^\delta((F,\C,\vp))$.
        So we assume $\rk(F') < \rk(F)$. Then $V' = \C$ or $0$.
        As $\vp$ is generically surjective, it cannot factor through $F'$.
        It follows that $\vp' = 0$ and $V'=0$.
        So $\mu^\delta((F',V',\vp')) = \mu(F')$.
        As $\delta > \delta_{\infty}$,
        we have $$\mu^\delta((F,\C,\vp)) = \frac{d+\delta}{r} > \max_{0<s<r} \left\{ \frac{d- (r-s) \cdot \mu_{\min}(E^\vee))}{s}
        \right\}\,.$$
        To show the semistability of $(F,\C,\vp)$, all we need to prove is that $$\mu(F') \leqslant \max_{0<s<r} \left\{ \frac{d- (r-s) \cdot \mu_{\min}(E^\vee))}{s}
        \right\}\,.$$
        Consider the following Harder-Narasimhan filtration of $F$:
        $$ 0 = F_0 \subset F_1 \subset \cdots \subset F_{l-1} \subset F_l = F$$
        with $F_i/F_{i-1}$ being semistable of slope $\mu_i$.
        It is enough to show that 
        $$\mu(F_1) \leqslant \max_{0<s<r}\left\{ \frac{d-(r-s) \cdot \mu_{\min}(E^\vee))}{s}
        \right\}\,.$$
        As $F/F_{l-1}$ is semistable, 
        from the composition $E^\vee \to F \to F/F_{l-1}$, we get 
        $\mu_l \geqslant \mu_{\min}(E^\vee)$.
        Consequently $\mu_i > \mu_{\min}(F)$ for any $i$.
        Using the short exact sequence 
        $$ 0 \to F_i/F_{i-1} \to F/F_{i-1} \to F/F_i \to 0$$
        and using induction, we see that 
        $$ \deg(F_1) \leqslant \deg(F) -\mu_{\min}(E^\vee) \cdot \rk(F/F_1)\,.$$
        This gives the required inequality for $\mu(F_1)$.

        For the converse, let $(F,\C,\vp)$ be $\delta$-semistable.
        If possible, assume $\vp$ is not generically surjective.
        Let $F' \subset F$ denote the image subsheaf of $\vp$.
        Let $F'$ has rank $r'$ and degree $d'$.
        Then $0 < r' < r$.
        From the semistability of $(F,\C,\vp)$, we have
        \begin{equation}\label{eq1}
            \frac{d'+\delta}{r'} \leqslant \frac{d+\delta}{r} \leqslant \frac{d-d'}{r-r'}\,.
        \end{equation}
        From the quotient $E \to F'$, we have 
        $\mu_{\min}(E^\vee) \geqslant \mu(F')$.
        It follows that 
        \begin{equation}\label{eq2}
            d - r \cdot \mu_{\min}(E^\vee) \geqslant \frac{dr'-rd'}{r'} \,.
        \end{equation}
        Using $\delta> \delta_{\infty}$ and \eqref{eq2}, 
        we get
        \begin{equation*}
            \frac{d'+\delta}{r'} \geqslant \frac{d'}{r'} + 
            \frac{d-r \cdot \mu_{\min}(E^\vee)}{r-r'}
            = \frac{d-d'}{r-r'}\,.
        \end{equation*}
        which is a contradiction to \eqref{eq1}.
        We conclude that $\vp$ must be generically surjective.
    \end{proof}
    Thus $Q(E,\alpha)$ parametrises pairs of the form 
    $(\vp: E^\vee \to F)$ such that $F$ is a vector bundle on $C$
    of the type $\alpha$ and $\vp$ is generically surjective.
    Such a point $(\vp: E^\vee \to F)$ gives rise to a short exact sequence 
    $[0 \to F^\vee \xrightarrow{\vp^\vee} E \to \coker(\vp^\vee) \to 0]$,
    which is a point of the Quot scheme parametrizing quotients of $E$
    of rank $(\rk(E)-r)$ and degree $(d+\deg(E))$.
    Moreover, this correspondence is one-to-one.
    We show that this correspondence is actually an isomorphism of schemes.
    For integers $r,d$, we write $\quot_{r,d}(E)$ to mean the Quot scheme
    which parametrises quotients of $E$ of rank $r$ and degree $d$.
    \begin{proposition}\label{proposition isomorphism Quot and Q}
        Let r and $d$ be integers such that $0<r<\rk(E)$
        and let $\alpha = (\rk(E)-r,d-\deg(E))$.
        There exists an isomorphism
        $$\mf f: \quot_{r,d}(E) \to Q(E,\alpha)$$
        which is compatible with the virtual fundamental classes, i.e.,
        $$ \mf f_*([\quot_{r,d}(E)]^\vir) = [Q(E,\alpha)]^\vir\,.$$
    \end{proposition}
    \begin{proof}
    Recall that $Q(E,\alpha)$ is a fine moduli space with a universal pair on $C \times Q(E,\alpha)$,
    $$ \Psi: q_C^*E^\vee \to \mF\,.$$
    Let
    $$0 \to \K \xrightarrow{\Phi} p_C^*E \to \Q \to 0$$
    denote the universal exact sequence on $C \times \quot_{r,d}(E)$. 
    Consider the dual map
    $$\Phi^\vee :  p_C^*E^\vee \to \K^\vee\,.$$
    For any closed point $[0 \to K \xrightarrow{\vp} E \to Q \to 0] \in \quot_{r,d}(E)$,
    the fiber of the map $\Phi^\vee$ over that point corresponds to the map
    $\vp^\vee: E^\vee \to K^\vee$,
    which is generically surjective and $K^\vee$ is locally free sheaf of type $\alpha$.
    As $Q(E,\alpha)$ is a fine moduli space, using Proposition \ref{proposition stability and gg}, we get a map
    $$\mf f: \quot_{r,d}(E) \xlongrightarrow{} Q(E,\alpha)\,$$
    such that 
    $$\mf f^*( q_C^*E^\vee \xlongrightarrow{\Psi} \mF) = 
    (p_C^*E^\vee \xlongrightarrow{\Phi^\vee} \K^\vee)\,.$$
    On the other hand, the dual morphism
    $$ \Psi^\vee: \mF^\vee \to q_C^*E$$
    is injective and the its cokernel is flat over $Q(E,\alpha)$.
    The universal property of the Quot scheme induces a map in the other direction, 
    $$\mf f' : Q(E,\alpha) \to \quot_{r,d}(E)$$
    such that 
    $$\mf f'^*(p_C^*E \to \Q) = (q_C^*E \to \coker(\Psi^\vee))\,.$$
    Clearly $\mf f'$ is an inverse of $\mf f$
    and hence $\mf f$ is an isomorphism.
    It is shown in \cite[Theorem 1]{MO07c} that the complex 
    $$ (R(p_Q)_*\HOM(\K,\Q))^\vee$$
    defines a  perfect obstruction theory on the Quot scheme,
    which gives rise to the virtual fundamental class
    $$[\quot_{r,d}(E)]^\vir \in H_\bullet(\quot_{r,d}(E))\,.$$
    Recall that the perfect obstruction theory on $Q(E,\alpha)$
    is given by $\eqref{equation POT on stable pairs}$.
    Comparing the perfect obstruction theories,
    it follows that 
    $$ \mf f_*([\quot_{r,d}(E)]^\vir)
    = [Q(E,\alpha)]^\vir\,.$$
    \end{proof}

    Now we consider the leftmost chamber, i.e., when $\delta$ is very small,
    we have the scheme $\Mdash^{0+\ss}_{\alpha,1}$ of $\mu^\delta$-stable pairs.
    For ease of notation, we denote this scheme by $P(E,\alpha)$, i.e.
    $$ P(E,\alpha) := \Mdash^{0+\ss}_{\alpha,1}\,.$$
    The following proposition gives a description for closed points in $P(E,\alpha)$.
    \begin{proposition}\label{proposition stability 0+ chamber}
        Let $\alpha = (r,d)$ with $0<r<\rk(E)$ and
        let $\delta < 1/r$.
        Then an $E^\vee$-pair $(F,\C,\vp)$ of type $(\alpha,1)$ is $\mu^\delta$-semistable iff
        $(F,\C,\vp)$ is $\mu^\delta$-stable iff
        $F$ is a slope-semistable vector bundle on $C$ and 
        $\vp$ is non-zero satisfying the following condition 
        \begin{equation}\label{condition small delta stability}
        \text{There is no proper subsheaf $G \subset F$ with $\mu(G)=\mu(F)$ such that $\im(\vp) \subset G$.}    
        \end{equation}
    \end{proposition}
    \begin{proof}
        Let $(F,\C,\vp)$ satisfies condition \eqref{condition small delta stability}, $F$ is slope-semistable and $\vp$ is nonzero.
        Consider any sub-object $(F',V',\vp') \subset (F,\C,\vp)$.
        If $\vp' = 0$ then $V' = 0$. 
        Using semistability of $F$, we have
        $$\mu^\delta((F',V',\vp')) = \mu(F')\leqslant \mu(F) < \mu^\delta((F,\C,\vp)) \,.$$
        Assume $\vp'\neq 0$. Then
        $$\mu^\delta((F,\C,\vp)) -\mu^\delta((F',V',\vp')) = 
        \frac{1}{rr'} \left(dr'-d'r + (r'-r)\delta \right)\,.$$
        By the semistability of $F$ and the condition \eqref{condition small delta stability},
        the quantity $(dr'-d'r) > 0$.
        Choosing $\delta < 1/r$, we clearly have
        $\mu^\delta((F,\C,\vp)) -\mu^\delta((F',V',\vp')) > 0$.
        Hence $(F,\C,\vp)$ is $\delta$-stable.

        For the converse, let $(F,\C,\vp)$ be $\delta$-semistable.
        Condition \eqref{condition small delta stability} is clear.
        First, we prove $F$ is slope-semistable.
        Let $F' \subset F$ be a subsheaf.
        Considering $(F',0,0)$ as a sub-object of $(F,\C,\vp)$,
        we get
        $$\mu(F') = \mu^\delta((F',0,0)) \leqslant \mu^\delta((F,\C,\vp)) = \mu(F) + \delta/r$$
        As $\delta<1/r$, we conclude $\mu(F') \leqslant \mu(F)$.
        Next, we show that $\vp \neq 0$.
        If possible, say $\vp =0$.
        Let $\rho : E^\vee \to G$ be a non-zero map
        such that $G$ is a semistable bundle of rank $r$ and slope
        $\mu(G) > \mu(F)$.
        Using the other direction of the proposition (which is already proved), 
        we see that the pair $(G,\C,\rho)$ is $\delta$-stable.
        Then we have the morphism $(0,\id): (G,\C,\rho) \to (F,\C,0)$ of $\delta$-stable objects.
        It follows that $\mu^\delta((G,\C,\rho))<\mu^\delta((F,\C,0))$,
        which is a contradiction.
        Hence $\vp \neq 0$. This completes the proof.
    \end{proof}
    Thus the moduli space $P(E,\alpha)$ parametrizes pairs of the form $\vp: E^\vee \to F$
    such that $F$ is a slope-semistable vector bundle on $C$ of the type $\alpha$ and 
    $\vp$ is non-zero satisfying condition \eqref{condition small delta stability}.
    This moduli space $P(E,\alpha)$ is used in \cite{BDW} in a different setting.
    When $\alpha=(r,d)$ with $r$ and $d$ being coprime,
    $P(E,\alpha)$ has a simpler description.
    Let $M_\alpha^s$ denote the moduli space of stable bundles of type $\alpha$ on $C$.
    Let $p: C \times M^s_\alpha \to C$ and 
    $q: C \times M^s_\alpha \to M^s_\alpha$ denote the projections.
    As $r$ and $d$ are coprime, the moduli space $M_\alpha^s$ 
    admits a universal sheaf, say $\G$, on $C \times M^s_\alpha$.
    It is easy to see that $P(E,\alpha)$ 
    is the virtual projective bundle $\mb P(Rq_*(p^*E \otimes \G))$ over $M^s_\alpha$.
    Moreover, when $d > r(2g-2-\mu_{\min}(E))$,
    we have the vanishing
    $H^1(E \otimes G)=0$ for any $G \in M^s_{\alpha}$.
    In this case $P(E,\alpha)$ is the smooth projective bundle
    $\mb P(q_*(p^*E\otimes \G))$ over $M^s_\alpha$.

    \section{Vertex Algebra Structure and Invariants}\label{section Vertex Algebra Structure and Invariants}
    We will define an enumerative invariant for the Quot scheme, 
    which takes place in a vertex algebra whose underlying vector 
    space is the homology of the stack $\Mdash$.
    We will follow the definition of homology of stacks 
    as in \cite[Definition 2.2]{Joy21}.
    It is defined as follows.
    To each higher stack $\mc X$, in \cite{Bla16}, 
    Blanc constructs a topological realization $|\mc X|$ which is a topological space. 
    The homology and cohomology of $\mc X$ are defined as the homology and cohomology of $|\mc X|$,
    $$ H^\bullet(\mc X):= H^\bullet(|\mc X|)\, \quad \quad 
    H_\bullet(\mc X) := H_\bullet(|\mc X|)\,.$$
    Let $\perf$ denote the moduli stack of perfect complexes on $\spec\,\C$, as in To\"en–Vaqui\'e \cite[Definition 3.28]{TV07}.
    Each perfect complex $\mc E$ on $\mc X$ corresponds to a map $\mc E: \mc X \to \perf$,
    which induces $|\mc E|: |\mc X| \to BU \times \mb Z$ and a K-theory class 
    $$ [[\mc E]] := |\mc E|^* (\U') \in K^0(|\mc X|)\,.$$
    The Chern classes of $\mc E$ are defined to be the Chern classes of $[[\mc E]]$.

    One crucial ingredient in the construction of the vertex algebra from $H_\bullet(\Mdash)$ is the complex $\acute{\mc E}$ on $\Mdash \times \Mdash$.
    We describe the construction of $\acute{\mc E}$ briefly.
    Recall the universal pair 
    $$\Theta: E^\vee \boxtimes \V \to \U$$ on $C \times \Mdash$.
    Let $\pi_{ij}$ denote the projection from 
    $C \times \Mdash \times \Mdash$ to the $i$-th and $j$-th factors
    and let $q_2,q_3$ denote the projections from $\Mdash \times \Mdash$ to its factors. 
    Following \cite[Definition 8.14]{Joy21}, we define 
    the $\acute{\EXT}$ complex on $\Mdash \times \Mdash$ 
    using the following triangle
    
    \begin{align}\label{equation definition of ext complex}
    \acute{\EXT} \longrightarrow 
    \substack{
    \displaystyle (\pi_{23})_*(\pi_{12}^*\U^\vee \otimes \pi_{13}^*\U)\\
    \displaystyle 
    \oplus \\ 
    \displaystyle (q_2^*\V^\vee \otimes q_3^*\V)}
    \xlongrightarrow{\wt \Theta_1 \oplus \wt \Theta_2}
    (\pi_{23})_*(\pi_{12}^*(E^\vee \boxtimes \V)^\vee \otimes \pi_{13}^*\U)
    \longrightarrow \acute\EXT[1]
    \end{align}
    \\
    
    \noindent{}
    where $\wt \Theta_1 := (\pi_{23})_*(\pi_{12}^*\Theta^\vee \otimes \pi_{13}^*\id)$
    and $\wt \Theta_2$ is defined using the following composition
    $$\pi_{23}^*(q_2^*\V^\vee \otimes q_3^*\V)
    \xlongrightarrow{\id \otimes (\pi_{13}^*\Theta)} 
    \pi_{23}^*q_2^*\V^\vee \otimes \pi_{13}^*(E \boxtimes \U)
    \cong \pi_{12}^*(E^\vee \boxtimes \V)^\vee \otimes \pi_{13}^*\U\,.$$
    We define the perfect complex $\acute{\mc E}$ to be the dual of the complex $\acute{\EXT}$, i.e.,
    $$\acute{\mc E} := \acute{\EXT}^\vee\,.$$
    For $(\alpha,e), (\alpha',e') \in K(\Adash)$,
    let $\acute{\mc E}_{(\alpha,e), (\alpha',e')}$ denote the restriction of $\acute{\mc E}$ to the component 
    $\Mdash_{\alpha,e} \times \Mdash_{\alpha',e'}$.
    We define the number
    $$\acute{\chi}((\alpha,e), (\alpha',e')) := \rk(\acute{\mc E}_{(\alpha,e), (\alpha',e')})\,.$$
    For $(\alpha,e) \in K(\Adash)$, define the shift
    $$\wh{H}_\bullet(\Mdash_{\alpha,e}) = H_{\bullet-2\acute{\chi}((\alpha,e),(\alpha,e))}(\Mdash_{\alpha,e})$$
    and
    $$\wh{H}_\bullet(\Mdash^\pl_{\alpha,e}) = H_{\bullet+2-2\acute{\chi}((\alpha,e),(\alpha,e))}(\Mdash^\pl_{\alpha,e})\,.$$
    Define the shifted homology 
    $$ \acute{V}_\bullet = \bigoplus_{(\alpha,e) \in K(\Adash)} \wh{H}_\bullet(\Mdash_{\alpha,e})\,.$$
    A graded vertex algebra structure is constructed on the shifted homology $\acute{V}_{\bullet}$ as follows.
    \begin{theorem}{\cite[Theorem 3.12]{Joy21}}\label{theorem vertex algebra structure}
    There is a vertex algebra structure on $\acute{V}_\bullet$ defined by
    \begin{enumerate}
        \item The unit element
        $1 \in \wh{H}_0(\Mdash)$ is given by the class
        $$ [0] \in H_0(\Mdash_{0,0})\, ,$$
        \item The translation operator $\acute{T} : \wh H_\bullet(\Mdash_{\alpha,e}) \longrightarrow \wh H_{\bullet+2}(\Mdash_{\alpha,e})$ is defined as follows.
        For $u \in \wh H_\bullet(\Mdash_{\alpha,e})$, we set
        $$\acute{T}(u) = \vartheta_*(t \boxtimes u)\, ,$$
        where $\vartheta: B\mb G_m \times \Mdash \to \Mdash$
        is the $B\mb G_m$-action on $\Mdash$ and
        $t \in H_2(B\mb G_m)$ is the dual of the first Chern class of the universal line bundle on $B\mb G_m$.
        \item The state-field correspondence is given by the following formula. 
        For $u \in \wh{H}_\bullet(\Mdash_{\alpha,e})$ and 
        $v \in \wh{H}_\bullet(\Mdash_{\alpha',e'})$, we set
        $$\qquad \acute Y(u,z)v = (-1)^{\acute{\chi}((\alpha,e), (\alpha',e'))} \cdot z^{\acute{\chi}_{\sym}((\alpha,e),(\alpha',e'))} \cdot 
        \oplus_* \left((e^{z\acute{T}} \otimes \id) ( c_{1/z}( \acute{\mc E}^{\sym}) \cap u \boxtimes v) \right) $$
        where $\acute{\chi}_{\sym}((\alpha,e),(\alpha',e')) = {\acute{\chi}((\alpha,e),(\alpha',e'))} + {\acute{\chi}((\alpha',e'),(\alpha,e))}$
        and $$\acute{\mc E}^{\sym} = \acute{\mc E} \oplus \sigma^*(\acute{\mc E}^\vee)\, ,$$ $\sigma : \Mdash \times \Mdash \to \Mdash \times \Mdash$ swaps the two factors.
    \end{enumerate}
    \end{theorem}
    It is shown in \cite[Theorem 4.8]{Joy21} that, 
    for $(\alpha,e) > 0$,
    the projection $\Mdash_{\alpha,e} \to \Mdash^\pl_{\alpha,e}$ induces an isomorphism 
    $$ \wh{H}_\bullet(\Mdash^\pl_{\alpha,e}) \cong \wh{H}_\bullet(\Mdash_{\alpha,e})/\im \acute{T}\,.$$
    Due to \cite{Bor86}, the vertex algebra on $\acute{V}_\bullet$ induces a graded Lie algebra structure on $\acute{V}_\bullet/ \im \acute{T}$, which is given by
    \begin{equation}\label{equation definition of bracket}
    [\bar{u},\bar{v}] = \Res_{z=0}(\acute Y(u,z)v) + \im \acute{T} \,,
    \end{equation}
    where $\bar u$ and $\bar v$ are images of $u, v \in \acute{V}_\bullet$.
    This defines a graded Lie algebra structure on the shifted homology
    $$ \wh{H}_\bullet(\Mdash^\pl_{>0}) =  \bigoplus_{(\alpha,e) >0} \wh{H}_\bullet(\Mdash^\pl_{\alpha,e})\,.$$
    Note that the subspace $\wh{H}_0(\Mdash^\pl_{>0}) = \bigoplus_{(\alpha,e)>0} \wh{H}_0(\Mdash^\pl_{\alpha,e})$ 
    is an ordinary Lie algebra.

    Let $\M$ denote the (higher) stack parametrizing objects 
    of the bounded derived category $D^b(\text{Coh}(C))$ of coherent sheaves on $C$. 
    For any $\alpha \in K(C)$, we have the open substacks of 
    slope-stable and slope-semistable sheaves of type $\alpha$,
    $$\M^{\st}_\alpha \subset \M^{\ss}_\alpha
    \subset \M^\pl_{\alpha}:= \M_\alpha/B\mb G_m\,.$$
    We can view the stacks $\M_{\alpha}, \M^\pl_{\alpha}$ and $\M^{\ss}_\alpha$
    as $\Mdash_{\alpha,0}, \Mdash^\pl_{\alpha,0}$ 
    and $\Mdash^{\delta-\ss}_{\alpha,0}$ (for any $\delta >0$)
    respectively.
    We can define the shifted homologies $\wh{H}_\bullet(\M_{\alpha})$
    and $\wh{H}_\bullet(\M^\pl_{\alpha})$ in the same way as we have done for $\Mdash$,
    and we can
    construct a similar vertex algebra structure on $V_\bullet$,
    where
    $$V_\bullet = \bigoplus_{\alpha \in K(C)} \wh{H}_\bullet(\M_{\alpha})\,.$$
    For $\alpha \in K(C)$ with $\alpha>0$, Joyce defined an enumerative invariant  in \cite[Theorem 7.63]{Joy21},
    $$ \sfM_\alpha := [\M^{\ss}_\alpha]_\inv \in \wh{H}_\bullet(\M^\pl_{\alpha})\,.$$
    When $\M^{\st}_\alpha = \M^{\ss}_\alpha$, this class coincides 
    with the fundamental class $[\M^{\ss}_\alpha]_{\fund}$
    of the smooth projective scheme $\M^{\ss}_\alpha$.
    In \cite[Theorem 8.24]{Joy21}, Joyce defined a similar invariant
    in the case of $\Mdash$, when $E$ is a line bundle.
    Generalizing this result for any vector bundle $E$,
    we have the following theorem.
    Recall that $\Mdash_{\alpha,1}^{\delta-\st}$ admits a virtual fundamental class $[\Mdash_{\alpha,1}^{\delta-\st}]^\vir$.
    Let the pushforward of this class along the inclusion $\Mdash_{\alpha,1}^{\delta-\st} \hookrightarrow \Mdash_{\alpha,1}^{\pl}$
    be also denoted by $[\Mdash_{\alpha,1}^{\delta-\st}]^\vir \in \wh H_\bullet(\Mdash_{\alpha,1}^{\pl})$.
    In Section \ref{section computing invariant}, we will compute (in \eqref{equation formula for chi})
    that if $\alpha=(r,d)$ and $\alpha'=(r',d')$ then
    \begin{equation*}
        \acute{\chi}((\alpha,e), (\alpha',e')) 
        = (1-g)rr'-\chi(E)\,r'e -\rk(E)\, d'e+ee'+rd'-r'd\,.
    \end{equation*}
    In particular,
    \begin{equation*}
        2-2\acute{\chi}((\alpha,0),(\alpha,0)) =2((g-1)r^2+1) = 2 \dim(M^s_\alpha)\,.
    \end{equation*}
    and
    \begin{equation*}
        2-2\acute{\chi}((\alpha,1),(\alpha,1)) =2(\rk(E)d+\chi(E)r-(1-g)r^2) = 2 \dim(\Mdash_{\alpha,1}^{\delta-\st})\,.
    \end{equation*}
    Hence, the ordinary Lie algebra 
    $\wh{H}_0(\M^\pl_{\alpha})$ is the correct degree component for the invariant $[\M^{\ss}_\alpha]_\inv$ to take place.
    Similarly, the invariant $[\Mdash_{\alpha,1}^{\delta-\st}]^\vir$ 
    lies in the Lie algebra $\wh{H}_0(\Mdash^\pl_{\alpha,1})$.
    
    \begin{theorem}\label{theorem wallcrossing formula}
        With the descriptions given above, we have the following.
        For any $\alpha\in K(C)$, $e\in \{0,1\}$ and any $\delta>0$,
        there is a unique class 
        $\sfMdash^\delta_{\alpha,e}:=[\Mdash_{\alpha,e}^{\delta-\ss}]_\inv$ in the  shifted homology group
        $\wh H_0(\Mdash_{\alpha,e}^{\pl})$.
        This class satisfies the following
        \begin{enumerate}
            \item If $\Mdash_{\alpha,e}^{\delta-\ss} = \Mdash_{\alpha,e}^{\delta-\st}$
            then $\sfMdash^\delta_{\alpha,e} = [\Mdash_{\alpha,e}^{\delta-\st}]^\vir$,
            the virtual fundamental class of the projective scheme $\Mdash_{\alpha,e}^{\delta-\st}$.
        
            \item For $\delta, \delta'>0$, we have the following wall-crossing formula
            \begin{align}\label{equation wall-crossing theorem}
            \nonumber
                \sfMdash^{\delta'}_{\alpha,1} = \sum_{\substack{n \geqslant 1, \rk(\alpha_i)>0\\ 
                \sum \alpha_i = \alpha, \sum e_i=1 \\
                 e_i \geqslant 0, \Mdash^{\delta-\ss}_{(\alpha_i,e_i)} \neq \varnothing
                }} \tilde U((\alpha_1,e_1),& \dots,(\alpha_n,e_n); \mu^{\delta'},\mu^{\delta}) \cdot \\
                & \left[ \Big[ \cdots \Big[ \Big[\sfMdash^\delta_{\alpha_1,e_1}, \sfMdash^\delta_{\alpha_2,e_2} \Big], 
                \sfMdash^\delta_{\alpha_3,e_3} \Big],
                \dots \Big],\sfMdash^\delta_{\alpha_n,e_n}\right]
            \end{align}
            in the Lie algebra $\wh H_{0}(\Mdash^{\pl})$.
            Here $\tilde U((\alpha_1,e_1),\dots (\alpha_n,e_n);\mu^\delta,\mu^{\delta'})$ 
            are combinatorial coefficients which are defined 
            in \cite[Theorem 3.12]{Joy21} and the sum is finite.
        \end{enumerate}
    \end{theorem}
    \begin{proof}
        The proof is same as the proof of \cite[Theorem 8.24]{Joy21}
        once we replace the line bundle $L$ with the vector bundle $E^\vee$ in the definition of $\Adash$.
        We only need to check that our setup satisfies all the required assumptions in \cite{Joy21}, i.e., assumptions 4.4 and 5.1-5.3.
        As in section 8.1.1 of \cite{Joy21},
        we can embed $\Adash$ into $Coh(X \times \mb P^1)$ using
        \cite{GP94}.
       Then similarly to \cite[Proposition 8.12]{Joy21}, we conclude 
       that $\Mdash^{\delta-\ss}_{\alpha,e}$ are of finite type.
        
        Now the assumption 4.4 follows similarly to \cite[Section 8.2.1]{Joy21},
        taking $\Bdash= \Adash$.
        Assumptions 5.1-5.3 also follows from \cite[Section 8.2.2-8.2.4]{Joy21},
        once we define the set of permissible classes
        $C(\Bdash)_\pe = \{(\alpha,e) \in C(\Bdash) : \rk(\alpha)>0 \text{ and } e\leqslant 1\}$.
    \end{proof}

    The inclusion $\iota : \M \hookrightarrow \Mdash$ sending $F$ to $(F,0,0)$ realizes $V_\bullet$ as a vertex subalgebra of $\acute{V}_\bullet$.
    The inclusion $\iota$ induces an inclusion $\iota^\pl: \M^\pl \hookrightarrow \Mdash^\pl$.
    For any $\alpha \in K(C)$ with $\alpha>0$
    and for any $\delta >0$, this inclusion
    identifies the moduli stack $\M^{\ss}_\alpha$ with the stack $\Mdash^{\delta-\ss}_{\alpha,0}$ and 
    hence identifies the sheaf invariant $\sfM_\alpha$ with the pair invariant $\sfMdash^\delta_{\alpha,0}$.
    With the help of the above theorem, one can define Joyce's invariant for $Q(E,\alpha)$.
    \begin{theorem}\label{theorem quot invariant}
        Let $\alpha \in K(C)$ with $0<\rk(\alpha)< \rk(E)$. 
        There are unique classes $[P(E,\alpha)]_\inv$ and $[Q(E,\alpha)]_\inv$ in the shifted homology 
        $\wh H_0(\Mdash_{\alpha,1}^{\pl})$.
        The class $[P(E,\alpha)]_\inv$ is equal to the pushforward of the virtual fundamental class $[P(E,\alpha)]^\vir$ under the inclusion $P(E,\alpha) = \Mdash^{0+\st}_{\alpha,1} \hookrightarrow \Mdash^\pl_{\alpha,1}$.
        The class $[Q(E,\alpha)]_\inv$ is equal to the pushforward of the virtual fundamental class $[Q(E,\alpha)]^\vir$ under the inclusion $Q(E,\alpha) = \Mdash^{\infty-\st}_{\alpha,1} \hookrightarrow \Mdash^\pl_{\alpha,1}$.
        Moreover, we have the following wall-crossing formula
        \begin{align}\label{equation wall crossing from quot to pair}
             \nonumber   [Q{(E,\alpha)}]_\inv
                = \sum_{\substack{n \geqslant 0, \sum \alpha_i = \alpha, \\ 
                \alpha_i=(r_i,d_i), r_i>0 \text{ for all }i,\\
                \mu_{\min}(E^\vee)\leqslant d_0/r_0<d_1/r_1\leqslant\cdots\leqslant d_n/r_n}}
                & \frac{(-1)^{n}}{\prod_{i=1}^l (a_i - a_{i-1})!} \cdot \\ & 
                \left[\Big[ \cdots \Big[\Big[[P(E,\alpha_0)]_\inv,\sfM_{\alpha_1} \Big], \sfM_{\alpha_2} \Big],\dots \Big],\sfM_{\alpha_n}\right]
            \end{align}
            in the Lie algebra $\wh H_{0}(\Mdash^{\pl})$.
        \end{theorem}
    \begin{proof}
    We choose $\delta< 1/r$ and $\delta' > \delta_\infty$.
    Then
    $\Mdash_{\alpha,1}^{\delta-\ss} = P(E,\alpha)$
    and $\Mdash_{\alpha,1}^{\delta'-\ss}= Q(E,\alpha)$.
    Applying Theorem \ref{theorem wallcrossing formula}, we get the invariants $[Q(E,\alpha)]_\inv$ and $[P(E,\alpha)]_\inv$,
    which coincide with $[Q(E,\alpha)]^\vir$ and $[P(E,\alpha)]^\vir$ respectively.
    Moreover, we get a wall-crossing formula of the form \eqref{equation wall-crossing theorem}.
    
    We compute the coefficients $\tilde U$. 
    First, we see that only one $e_i$ can be 1 and the others must be 0.
    So we need to consider the sequence of types
    $$(\alpha_1,0), \dots, (\alpha_j,0), (\alpha_{j+1},1), (\alpha_{j+2},0),\dots,(\alpha_n,0)$$
    where $\alpha_1 + \cdots + \alpha_n = \alpha = (r,d)$
    and $1 \leqslant j \leqslant n$.
    We use \cite[Definition 3.10]{Joy21} to calculate the coefficient
    $$U((\alpha_1,0), \dots, (\alpha_j,0), (\alpha_{j+1},1),(\alpha_{j+2},0),\dots,(\alpha_n,0); \mu^\delta,\mu^{\delta'})\,.$$
    We use the same notations as in Definition 3.10 of \cite{Joy21}
    to define the numbers $l,m,a_i,b_i,\beta_i,\gamma_i$, 
    setting $\mu^\delta = \tau$ and $\mu^{\delta'}=\tilde{\tau}$.
    The condition $\tilde{\tau}(\gamma_i)= \tilde{\tau}(\alpha_1 + \cdots + \alpha_n)$ for all $i$, forces all $\gamma_i$ to contain $(\alpha_{j+1},1)$. 
    Therefore, $l=1$, $\gamma_1 = \beta_1 +\cdots + \beta_m$ and we have
    $$U((\alpha_1,0), \dots, (\alpha_j,0), (\alpha_{j+1},1),(\alpha_{j+2},0),\dots,(\alpha_n,0); \mu^\delta,\mu^{\delta'}) = 
    S(\beta_1,\dots,\beta_m) \cdot \prod_{i=1}^n \frac{1}{(a_i-a_{i-1})!} \,.$$
    It is easy to see that
    $$ S(\beta_1,\dots,\beta_m) = 
    \begin{cases}
        \frac{(-1)^{n-1-j}}{\prod_{i=1}^l(a_i-a_{i-1})!}
        & \text{ if } d_1/r_1 \geqslant \dots \geqslant d_j/r_j > d_{j+1}/r_{j+1} < d_{j+2}/r_{j+2} \leqslant \cdots \leqslant d_n/r_n\, ,\\
        0 \, , &\text{ otherwise.}
    \end{cases}$$
    Comparing the coefficients $U$ and $\tilde U$, one can take
    \begin{align*}
        \tilde{U}((\alpha_1,0), \dots, (\alpha_j,0), & (\alpha_{j+1},1),(\alpha_{j+2},0),\dots,(\alpha_n,0); \mu^\delta,\mu^{\delta'}) \\
        &= \begin{cases}
            \frac{(-1)^{n-1}}{\prod_{i=1}^l (a_i - a_{i-1})!} \,, 
            & \text{ if $j=0$ and } d_1/r_1< d_2/r_2\leqslant \dots \leqslant d_n/r_n \,,\\
            0 \,, & \text{ otherwise ,}
        \end{cases}
    \end{align*}    
    where $a_0,a_1,\dots,a_l$ are as in the statement of the theorem.
    Moreover, we recall that $P(E,\alpha_1)$ is non-empty only when $d_1/r_1 \geqslant \mu_{\min}(E^\vee)$.
    Therefore, we get the following wall-crossing formula
    \begin{align*}
                [Q{(E,\alpha)}]_\inv
                = \sum_{\substack{n \geqslant 1, \sum \alpha_i = \alpha, \\ 
                \alpha_i=(r_i,d_i), r_i>0 \text{ for all }i,\\
                \mu_{\min}(E^\vee) \leqslant d_1/r_1<d_2/r_2\leqslant\cdots\leqslant d_n/r_n}}
                & \frac{(-1)^{n-1}}{\prod_{i=1}^l (a_i - a_{i-1})!} \cdot \\
                &\left[\Big[ \cdots \Big[ \Big[ [P(E,\alpha_1)]^\vir,\sfM_{\alpha_2} \Big], \sfM_{\alpha_3} \Big],\dots \Big],\sfM_{\alpha_n} \right] \,.
            \end{align*}
    The statement of the theorem follows by shifting the index number.
    \end{proof}

For applications in subsequent sections, we would express the class $[P(E,\alpha)]_\inv$ in terms of the classes $\sfM_{\alpha'}$.
Recall the category $\Adash$ of pairs defined in Definition \ref{Definition category of pairs}.
For each $p \in \mb Q$, let $\Adash_p$ be the full subcategory consisting of $E^\vee$-pairs $(F,V,\vp)$ with $F$ slope-semistable of slope $p$ or $F=0$.
Also, let us define 
$$ C(\Adash_p):= \{(\alpha,e) \in K(\Adash): (\alpha,e)>0 \text{ and slope of $\alpha$ is } p\}\,.$$
Fix $p \in \mb Q$. Following \cite[Section 13.2]{JS12}, 
for each $\delta \in \mb R$,
we define a stability condition $\mu_p^\delta$ on the category $\Adash_p$ as follows.
For a class $(\alpha,e) = ((r,d),e)\in C(\Adash_p)$, define
$$ \mu_p^\delta(\alpha,e) = 
\begin{cases}
            \frac{\delta e}{r}\,. \quad  & \textnormal{ if } r>0\,,\\
            0 & \textnormal{ if } r=0\,.
        \end{cases} $$
We notice that when $\delta>0$ is such that $\mu^\delta = \mu^{0+}$ (e.g., when $0 < \delta <1/r$),
the stability conditions $\mu_p^\delta$ are equivalent to the restriction of $\mu^{\delta}$ to $\Adash_p$.
The stability condition $\mu_p^0$ is a trivial condition.
When $\delta<0$, the stability conditions $\mu_p^\delta$ are all mutually equivalent.

One can verify that the family $\{\mu_p^\delta\}_\delta$ satisfies 
\cite[Assumptions 5.1-5.3]{Joy21} as done in \cite[Section 13]{JS12}. 
So we can apply the Wall-crossing formula \cite[Theorem 5.9]{Joy21} for the stability conditions $\mu_p^\delta$ on the category $\Adash_p$.
\begin{proposition}\label{proposition wall-crossing of pairs}
    Let $\alpha=(r,d) \in K(C)$ with $0<r<\rk(E)$.
    Then we have the following formula
    \begin{align*}
    [P(E,\alpha)]_\inv =
     \sum_{\substack{m \geqslant 1, \sum \alpha_i = \alpha, \\ 
                \alpha_i=(r_i,d_i), r_i>0 \text{ for all }i,\\
                d_i/r_i=d/r \text{ for all }i}}
                & \frac{(-1)^{m}}{m!} \cdot
    \left[ \left[ \cdots \left[ \left[\mathsf e^{((0,0),E^\vee)},\sfM_{\alpha_1} \right], \sfM_{\alpha_2} \right],\dots \right],\sfM_{\alpha_m}\right]\, ,
    \end{align*}
    where $\mathsf e^{((0,0),E^\vee)}$ is the class of the single point set $\{(0,\C,0)\}$ in $\wh H_{\bullet}(\Mdash^{\pl})$.
    \end{proposition}
    \begin{proof}
        We take $p=d/r$ and consider the category $\Adash_p$ and the stability conditions $\mu_p^\delta$ defined above.
        We apply the wall-crossing formula from \cite[Theorem 5.9]{Joy21} taking $\tau = \mu_p^\delta$ with $\delta =-1$ and $\tilde{\tau} = \mu_p^{\delta'}$ with $\delta'<1/r$.
        We have the following formula
        \begin{align*}
        [\Mdash^{\ss}_{\alpha,1}(\mu_p^{\delta'})]_\inv =
        &\sum_{\substack{m \geqslant 1, \sum \alpha_i = \alpha,
                \sum e_i = 1 \\ 
                \alpha_i=(r_i,d_i), r_i>0 \text{ for all }i,\\
                \Mdash^{\ss}_{(\alpha_i,e_i)}(\mu_p^{-1}) \neq \varnothing,  d_i/r_i= p \text{ for all }i
                }}
                \tilde U((\alpha_1,e_1),\dots,(\alpha_n,e_n);\mu_p^{\delta},\mu_p^{\delta'}) \cdot \\   
        & \qquad \quad \left[\Big[ \cdots \Big[ [\Mdash^{\ss}_{\alpha_1,e_1}(\mu_p^{-1})]_\inv, [\Mdash^{\ss}_{\alpha_2,e_2}(\mu_p^{-1})]_\inv \Big],\dots \Big],[\Mdash^{\ss}_{\alpha_m,e_m}(\mu_p^{-1})]_\inv\right] \,.
    \end{align*}
        As $\delta' <1/r$, using Proposition \ref{proposition stability 0+ chamber},
        it follows that
        $\Mdash^{\ss}_{\alpha,1}(\mu_p^{\delta'})= P(E,\alpha)$ .
        Also, $\Mdash^{\ss}_{\alpha,1}(\mu_p^{-1}) = \varnothing$ whenever $\alpha \neq (0,0)$.
        And $\Mdash^{\ss}_{\alpha,0}(\mu_p^{-1})$ is same as the moduli space $\M^{\ss}_\alpha$ of semistable sheaves of type $\alpha$.
        So the formula becomes the following,
        \begin{align*}
        [P(E,\alpha)]_\inv =
        \sum_{\substack{m \geqslant 1, \sum \alpha_i = \alpha, 
                0 \leqslant j \leqslant m\\ 
                \alpha_i=(r_i,d_i), r_i>0 \text{ for all }i,\\
                \Mdash^{\ss}_{\alpha_i} \neq \varnothing,  d_i/r_i= p \text{ for all }i}}
                \tilde U((\alpha_1,0),\dots,(\alpha_j,0),(0,1),(\alpha_{j+1},0),\dots, (\alpha_n,0); \mu_p^{\delta},\mu_p^{\delta'}) \cdot & \\      
         \left[ \Big[ \cdots \Big[ 
        \sfM_{\alpha_1}, \sfM_{\alpha_2} \Big],\dots \Big],
        \sfM_{\alpha_j} \Big], [\Mdash^{\ss}_{0,1}(\mu_p^{-1})]_\inv \Big],\sfM_{\alpha_{j+1}} \Big],
        \dots \Big],\sfM_{\alpha_m} \right]\,.&
    \end{align*}
        The coefficients $\tilde U$ can be computed similarly as in the proof of Theorem \ref{theorem quot invariant}.
        One can take 
        $$\tilde U((\alpha_1,0),\dots,(\alpha_j,0), (0,1),(\alpha_{j+1},0),\dots, (\alpha_m,0); \mu_p^{-1},\mu_p^{\delta'}) = 
        \begin{cases}
            \frac{(-1)^m}{m!}, &\text{ if } j=0,\\
            0, &\textnormal{otherwise.}
        \end{cases}$$
    Since $\Mdash^{\ss}_{(0,0),1}(\mu_p^{-1})$ is a single point set $\{(0,\C,0)\}$, we write the class $[\Mdash^{\ss}_{(0,0),1}(\mu_p^{-1})]_\inv$ as $\mathsf e^{((0,0),E^\vee)}$.
    This proves the theorem.
    \end{proof}

We translate the theorems above in terms of Quot schemes.
For the Quot scheme $\quot_{r,d}(E)$, recall the isomorphism
$\mf f : \quot_{r,d}(E) \to Q(E,\alpha)$
with $\alpha=(\rk(E)-r,d-\deg(E))$, constructed in Proposition \ref{proposition isomorphism Quot and Q}.
Composing this isomorphism with the inclusion $Q(E,\alpha) \subset \Mdash^\pl_{\alpha,1}$,
we get an open embedding, which we call $\mf f$ as well,
$$\mf f: \quot_{r,d}(E) \to \Mdash^\pl_{\alpha,1}\,.$$
\begin{theorem}\label{theorem wall crossing for quot scheme}
        Let $r$ and $d$ be integers such that $0<r<\rk(E)$.
        Define $\alpha=(\rk(E)-r,d-\deg(E))$.
        There is a unique class 
        $$[\quot_{r,d}(E)]_\inv \in \wh H_0(\Mdash^\pl_{\alpha,1})\,,$$
        which is equal to the pushforward class
        $$\mf f_*[\quot_{r,d}(E)]^\vir\,.$$
        The invariant satisfies the wall-crossing formula
        \begin{align}\label{equation wall crossing for quot to e^001}
        \nonumber
    [\quot_{r,d}(E)]_\inv
            =\sum_{\substack{n\geqslant 1, \sum \alpha_i = \alpha, \\ 
            \alpha_i=(r_i,d_i), r_i>0 \text{ for all }i,\\
            \mu_{\min}(E^\vee) \leqslant d_1/r_1\leqslant\cdots\leqslant d_n/r_n}} &
            \frac{(-1)^{n}}{\prod_{i=1}^l (a_i - a_{i-1})!} \cdot \\
            &  \left[\left[ \cdots
     \left[ \left[\mathsf e^{((0,0),E^\vee)},\sfM_{\alpha_1} \right], \sfM_{\alpha_2} \right],\dots \right], \sfM_{\alpha_n} \right]
    \end{align}
    in $\wh H_{0}(\Mdash^{\pl})$,
    where $\mathsf e^{((0,0),E^\vee)}$ is the class of the single point 
    set $\{(0,\C,0)\}$ in $\wh H_{0}(\Mdash^{\pl})$
    and the numbers $0=a_0<\cdots<a_l=n$ are defined such that for any $0<i<n$, 
    we have $d_i/r_i<d_{r+1}/r_{i+1}$ if and only if $i=a_k$ for some $0<k<l$.
    \end{theorem}
    \begin{proof}
        Indeed, we define 
        $$[\quot_{r,d}(E)]_\inv := [Q(E,\alpha)]_\inv $$
        with $\alpha=(\rk(E)-r,d-\deg(E))$.
        The fact that $[\quot_{r,d}(E)]_\inv = \mf f_*[\quot_{r,d}(E)]^\vir$ 
        is clear from Proposition \ref{proposition isomorphism Quot and Q}.

        For the wall crossing formula, we substitute the formula from Proposition 
        \eqref{proposition wall-crossing of pairs} in the wall-crossing formula 
        \eqref{equation wall crossing from quot to pair} to get the following formula
    \begin{align*}
    [Q(E,\alpha)]_\inv =
           & \sum_{\substack{n \geqslant 0, \sum \alpha_i = \alpha, \\ 
                \alpha_i=(r_i,d_i), r_i>0 \text{ for all }i,\\
                \mu_{\min}(E^\vee) \leqslant d_0/r_0<d_1/r_1\leqslant\cdots\leqslant d_n/r_n}}
                \frac{(-1)^{n}}{\prod_{i=1}^l (a_i - a_{i-1})!} \cdot 
            \sum_{\substack{m \geqslant 1, \sum \beta_j = \alpha_0, \\ 
                \beta_j=(d'_j,r'_j), r'_i>0 \text{ for all }j,\\
                d'_j/r'_j=d_0/r_0 \text{ for all }j}}
                \frac{(-1)^{m}}{m!} \cdot \\
    \nonumber   & \qquad \quad
    \Big[ \Big[ \cdots \Big[
    \left[\left[ \cdots
    \left[ \left[e^{((0,0),E^\vee)},\sfM_{\beta_1} \right], \sfM_{\beta_2} \right],\dots \right],\sfM_{\beta_m}\right], \sfM_{\alpha_1} \Big],
    \cdots \Big], \sfM_{\alpha_n} \Big]
    \end{align*}
    in $\wh H_{0}(\Mdash^{\pl})$.
    The numbers $a_i$ are defined as in the statement of the theorem.
    Replacing $n$ for $n+m$, we get our formula as in the statement of the theorem.
    \end{proof}

\vspace{.3cm}

\section{Computing invariant for $\quot_{\rk(E)-1,d}(E)$}\label{section computing invariant}
In this section, we compute the invariant 
$[Q(E,\alpha)]_\inv$ in $H_\bullet(\Mdash_{\alpha,1}^\pl;\mb C)$ for $\rk(\alpha) =1$.
We will use an explicit description of $H_\bullet(\Mdash_{\alpha,1}^\pl;\mb C)$ given by \cite{gros20}.
The computations in this section are very similar to those in \cite[Theorem 4.3]{Bu23}. 
However, there are many differences in the numerical values.
We therefore include the computations for the convenience of the reader.

Consider a symplectic basis for the cohomology ring of $C$ as 
$$\{1=\e_{1,0}, \e_{1,1}, \dots,\e_{2g,1},\e_{1,2}\}\, ,$$
so that $\e_{*,k} \in H^k(C;\mb C)$.
Let $J$ denote the index set of the basis, i.e.,
$$J = \{(1,0),(1,1),\dots,(2g,1),(1,2)\}\,.$$
Let $\{\e^\vee_{j,k} : (j,k)\in J\}$ denote the dual basis of $H_\bullet(C;\mb C)$.
There is a universal perfect complex $\W_\alpha$ on $C \times \M_\alpha$.
We define the cohomology classes
\begin{align}\label{equation definition S_j,k,l classes}
\qquad S^\alpha_{j,k,l} := \ch_l(\W_\alpha) \setminus \e^\vee_{j,k} 
 \qquad \text{ for }(j,k)\in J \text{ and } l \geqslant k/2\,, 
\end{align}
where 
$$\setminus : H^{2l}(C \times \M_\alpha;\mb C) \otimes H_k(C;\mb C) \to H^{2l-k}(\M_\alpha;\mb C) $$
denotes the slant product.
Theorem 4.15 of \cite{gros20} says that there is an isomorphism
$$H^\bullet(\M_\alpha;\mb C) \cong \mb C[S^\alpha_{j,k,l}: (j,k)\in J, l>k/2]\,,$$
as graded commutative algebras with $\deg(S^\alpha_{j,k,l}) = 2l-k$.
Let $\perf_e$ denote the moduli stack of perfect complexes on $\spec \, \C$ of rank $e$.
The universal pair $\Theta: E^\vee \boxtimes \V \to \U$ on $C \times \Mdash$ induces a map
$$ \Pi_{\alpha,e}^\U \times \Pi_{\alpha,e}^\V : \Mdash_{\alpha,e} \to \M_\alpha \times \perf_e$$
which is an $\mb A^1$-homotopy equivalence.
This gives an isomorphism, 
\begin{align*}
H^\bullet(\Mdash_{\alpha,e};\mb C) & \cong H^\bullet(\M_\alpha;\mb C) \otimes H^\bullet(\perf_e;\mb C)\\ 
& \cong \mb C[S^\alpha_{j,k,l}: (j,k)\in J, l>k/2] \otimes \mb C[R^e_{l} : l>0]\,.
\end{align*}
as graded algebras, where $R^e_l$ is defined as $\ch_l(\V_e)$ and $\deg(R^e_l) = 2l$.
Let us define $$\acute J = J \sqcup \{(+,0)\}$$ and
for $(j,k)\in \acute J$, define the classes $S^{(\alpha,e)}_{j,k,l}$
in $H^{2l-k}(\Mdash_{\alpha,e};\mb C)$
as the following,
\begin{align*}
    S^{(\alpha,e)}_{j,k,l} = 
    \begin{cases}
        (\Pi_{\alpha,e}^\U)^*S^{\alpha}_{j,k,l} \qquad &\text{ if }(j,k) \in J\,,\\
        (\Pi_{\alpha,e}^\V)^*R^e_l &\text{ if }(j,k) = (+,0)\,.\\
    \end{cases}
\end{align*}
It follows that 
$$H^\bullet(\Mdash_{\alpha,e};\mb C) \cong \mb C[S^{(\alpha,e)}_{j,k,l}: (j,k)\in \acute J, l>k/2]\,.$$
We will compute $\ch(\acute{\mc E})$ in terms of the classes $S^{(\alpha,e)}_{j,k,l}$. 
For $(j,k), (j',k') \in J$,
define the numbers
$${M}_{j,k}^{j',k'} :=
        \int_C \epsilon_{j,k} \cup \epsilon_{j',k'} \cup \td(C)\,,$$
where $\td (C) = 1 + (1-g) \, \epsilon_{1,2}$
is the Todd class of $C$.
From the definition of $\acute \EXT$ complex (given by \eqref{equation definition of ext complex})
on $\Mdash \times \Mdash$, we have
\begin{align*}
    \ch(\acute \EXT) & = 
    \ch\left((\pi_{23})_*(\pi_{12}^*\U^\vee \otimes \pi_{13}^*\U) \right)
    + \ch(q_2^*\V^\vee \otimes q_3^*\V) -
    \ch\left((\pi_{23})_*(\pi_{12}^*(E^\vee \boxtimes \V)^\vee \otimes \pi_{13}^*\U)\right)
\end{align*}
We restrict everything on the component $\Mdash_{\alpha,e} \times \Mdash_{\alpha',e'}$.
Using Grothendieck-Riemann-Roch, the first term becomes
\begin{align*}
    & \ch\left((\pi_{23})_* (\pi_{12}^*\U_{\alpha,e}^\vee \otimes \pi_{13}^*\U_{\alpha',e'}) \right) 
     = (\pi_{23})_* (\ch((\pi_{12}^*\U_{\alpha,e}^\vee \otimes \pi_{13}^*\U_{\alpha',e'}) \cup \pi_1^*\td(C))) \\
    & \quad = (\pi_{23})_* \left(
    \left(\sum_{\substack{(j,k) \in J \\ l \geqslant k/2}} (-1)^l \e_{j,k} \boxtimes S^{(\alpha,e)}_{j,k,l} \boxtimes 1 \right) \cdot 
    \left(\sum_{\substack{(j',k') \in J \\ l' \geqslant k/2}} \e_{j',k'} \boxtimes  1 \boxtimes S^{(\alpha',e')}_{j',k',l'}\right) \cdot 
    \big( \td(C) \boxtimes 1 \boxtimes 1 \big)
    \right) \\
    & \quad  = \sum_{\substack{(j,k) \in J,\, l \geqslant k/2 \\ (j',k') \in J,\, l' \geqslant k/2}} (-1)^l \ {M}_{j,k}^{j',k'} 
    \ S^{(\alpha,e)}_{j,k,l} \boxtimes S^{(\alpha',e')}_{j',k',l'}
\end{align*}
Calculating the other terms similarly, we have
\begin{align*}
   \ch(\acute \EXT_{(\alpha,e), (\alpha',e')}) 
    & = \sum_{\substack{(j,k) \in J,\, l \geqslant k/2 \\ (j',k') \in J,\, l' \geqslant k/2}} (-1)^l \ {M}_{j,k}^{j',k'} 
    \ S^{(\alpha,e)}_{j,k,l} \boxtimes S^{(\alpha',e')}_{j',k',l'}
    + \sum_{l,l'} (-1)^l S^{(\alpha,e)}_{+,0,l} \boxtimes S^{(\alpha',e')}_{+,0,l'} \\ 
    & \qquad \qquad - \left( \sum_{l} \int_C \ch(E) \cup \e_{j,k} \cup \td(C) \right) \cdot S^{(\alpha,e)}_{+,0,l} \boxtimes S^{(\alpha',e')}_{j',k',l'} \\
    & = \sum_{\substack{(j,k) \in \acute J,\, l \geqslant k/2 \\ (j',k') \in \acute J,\, l' \geqslant k/2}} (-1)^l \ \acute{M}_{j,k}^{j',k'} \cdot
    S^{(\alpha,e)}_{j,k,l} \boxtimes S^{(\alpha',e')}_{j',k',l'}
\end{align*}
where we define $\acute M_{j,k}^{j',k'}$ as the following
\begin{equation}
    \acute{M}_{j,k}^{j',k'} =
    \begin{cases}
    \displaystyle
        M_{j,k}^{j',k'} &
        \text{ if } (j, k) \neq ({+}, 0) \neq (j', k')\,, \\
        \displaystyle
        -\int_C \ch(E) \cup \epsilon_{j',k'} \cup \td(C) &
        \text{ if } (j, k) = ({+}, 0) \neq (j', k')\,, \\
        0 &
        \text{ if } (j, k) \neq ({+}, 0) = (j', k') \,, \\
        1 &
        \text{ if } (j, k) = ({+}, 0) = (j', k')\,,
    \end{cases}
\end{equation}
So the chern character of $\acute{\mc E}$ can be written as
\begin{equation*}
    \ch(\acute{\mc E}_{(\alpha,e), (\alpha',e')}) 
    = \sum_{\substack{(j,k) \in \acute J,\, l \geqslant k/2 \\ (j',k') \in \acute J,\, l' \geqslant k/2}} (-1)^{l'-(k+k')/2} \ \acute{M}_{j,k}^{j',k'} \cdot
     S^{(\alpha,e)}_{j,k,l} \boxtimes S^{(\alpha',e')}_{j',k',l'}
\end{equation*}
In particular, for $\alpha=(r,d)$ and $\alpha'=(r',d')$,
we have
\begin{align}\label{equation formula for chi}
    \nonumber
    \acute{\chi}((\alpha,e), (\alpha',e')) 
    & = \ch_0(\acute{\mc E}_{(\alpha,e), (\alpha',e')}) 
     = \sum_{\substack{(j,k) \in \acute J,\, l= k/2 \\ (j',k') \in \acute J,\, l' = k/2}} (-1)^{l'-(k+k')/2} \ \acute{M}_{j,k}^{j',k'} \cdot
     S^{(\alpha,e)}_{j,k,l} \boxtimes S^{(\alpha',e')}_{j',k',l'}\\
     & = (1-g)rr'-\chi(E)\,r'e -\rk(E)\, d'e+ee'+rd'-r'd\,.
\end{align}
For the homology of the stack $\Mdash$, we define an isomorphism
$$H_\bullet(\Mdash_{\alpha,e};\mb C) \cong \ex^{(\alpha,e)} \cdot \mb C[s_{j,k,l}: (j,k)\in \acute J, l>k/2]\,,$$
where $\ex^{(\alpha,e)}$ is a formal symbol keeping track of the type $(\alpha,e)$ and
$s_{j,k,l}$ is of degree $2l-k$, with the following pairing formula
\begin{align}
    \label{equation definition s_jkl by dual}
    & \left\langle \biggl(
        \prod_{ \substack{ (j, k) \in \acute{J} \\ l > k/2 } }
        \left(S^{(\alpha,e)}_{j,k,l}\right)^{m_{j,k,l}}
    \biggr) \,,
    \biggl(
        \ex^{(\alpha,e)} \cdot
        \prod_{ \substack{ (j, k) \in \acute{J} \\ l > k/2 } }
        s_{j,k,l}^{m'_{j,k,l}}
    \biggr) \right\rangle
    \\  \nonumber
    &
    = \begin{cases}
        \displaystyle
        \prod_{ \substack{
            (j, k), (j', k') \in \acute{J} \\
            l > k/2, \ l' > k'/2, \\
            (j',k',l') \prec (j,k,l) \\
            k, k' \text{ odd}
        } } (-1)^{m_{j,k,l} \, m'_{j',k',l'}}
        \cdot
        \prod_{ \substack{ (j, k) \in J \\ l > k/2 } }
        m_{j,k,l}!, &
        \text{if } m_{j,k,l} = m'_{j,k,l} \text{ for all } j,k,l, \\
        0, &
        \text{otherwise}\,.
    \end{cases}
\end{align}
Here, the order of the products is determined by
the total order $\prec$
on the set of all $(j,k,l)$ with $(j,k) \in \acute{J}$ and $l>k/2$,
given by $(j',k',l') \prec (j,k,l)$
if $l'<l$, or $l'=l$ and $k'<k$, or
$l'=l$, $k'=k$ and $j'<j$.

Recall (from Theorem \ref{theorem vertex algebra structure}) that the shifted homology $\acute V_\bullet$ has a vertex algebra structure, with a translation operator $\acute{T}$ and a
state-field correspondence $\acute{Y}$.
Computing directly or using \cite[Equation 2.89]{Bu23}, 
we write the operators $\acute{T}$ and $\acute{Y}$ explicitly 
in terms of the classes $\ex^{(\alpha,e)} \cdot s_{j,k,l}$ 
as follows. 
Let $\alpha = (r, d) \in K(C)$, \ $\alpha' = (r', d') \in K (C)$
and $e,e'$ such that $(\alpha,e)>0$ and $(\alpha',e')>0$.
Define
$$
    \partial_{j,k,l} =
    \begin{cases}
        r, & (j, k, l) = (1, 0, 0), \\
        d, & (j, k, l) = (1, 2, 1), \\
        e, & (j, k, l) = ({+}, 0, 0), \\
        \frac{\p}{\p s_{j,k,l}}, & l > k/2,
    \end{cases}
$$
and also define $\p'_{j,k,l}$ similarly. Then
Let $p, p' \in \mb C[s_{j,k,l}:(j,k) \in \acute{J}, l>k/2]$.
\begin{align}
    \acute{T}(\ex^{(\alpha,e)} \cdot p) = 
    & \ex^{(\alpha,e)} \cdot 
    \left( rs_{1,0,1} + ds_{1,2,2} + e s_{+,0,1} + \sum_{(j,k) \in \acute{J}} s_{j,k,l+1} \p s_{j,k,l}
    \right) \, p
\end{align}
and
\begin{align}\label{equation description of Y}
    \nonumber
    \acute{Y}(\ex^{(\alpha,e)} \cdot p,z)\, \ex^{(\alpha',e')} & \cdot p' = 
     \ex^{(\alpha+\alpha',e+e')} \cdot (-1)^{\acute{\chi}((\alpha,e), (\alpha',e'))} \cdot 
    z^{\acute{\chi}_{\sym}((\alpha,e),(\alpha',e'))} \cdot
    \exp(z \acute{T} ) \circ \\
    & \exp \left[ - \sum_{(j,k) \in \acute{J}} (-1)^l ( l+l'-(k+k')/2 -1)! z^{-(l+l'-((k+k')/2))} \cdot \right.\\ 
    \nonumber
    & \quad \qquad \left(
            \acute{M}_{j,k}^{j',k'} +
            (-1)^{kk' + (k+k')/2} \, \acute{M}_{j',k'}^{j,k}
            \right) \p_{j,k,l}\p'_{j',k',l'} \Biggr]
    (p \cdot p') \vert_{s_{j,k,l}= s'_{j,k,l}} \,.
\end{align}
    We recall that the projection $\Pidash: \Mdash_{\alpha,e} \to \Mdash^\pl_{\alpha,e}$ induces an isomorphism 
    $$ \wh{H}_\bullet(\Mdash^\pl_{\alpha,e}) \cong \wh{H}_\bullet(\Mdash_{\alpha,e})/\im \acute{T}$$
    and there is a graded Lie algebra structure on the shifted homology
    $$ \wh{H}_\bullet(\Mdash^\pl_{>0}) =  \bigoplus_{(\alpha,e) >0} \wh{H}_\bullet(\Mdash^\pl_{\alpha,e})\,.$$
    The subspace $\wh{H}_0(\Mdash^\pl_{>0}) = \bigoplus_{(\alpha,e)>0} \wh{H}_0(\Mdash^\pl_{\alpha,e})$ 
    is an ordinary Lie algebra, where the class $[Q(E,\alpha)]_\inv$ lies.
With the notations above, we can now compute the invariant of the Quot scheme $Q_{\rk(E)-1,d}(E)$.
\begin{theorem}\label{theorem invariant quot N-1}
    Let $N =\rk(E)$ and $\nu = \chi(E)$.
    For any degree $d$, define $\du = d- \deg(E)$.
    The invariant 
    $$[Q_{\rk(E)-1,d}(E)]_\inv \quad  \in \quad \wh H_0(\Mdash_{(1,\du),1})/\im \acute T$$ is given by
    \begin{equation*}
    [Q_{\rk(E)-1,d}(E)]_\inv = 
    \res_{z=0} \left( \frac{1}{z^{\nu+N\du}} \cdot \rho(z) \cdot \sigma\left(\frac{N}{z} -s_{1,2,2}\right)\right) + \textnormal{im} \acute{T}\,,
    \end{equation*}
    where $\rho$ and $\sigma$ are defined as
\begin{align}\label{equation definition rho}
    \rho (z) & = \exp \biggl(
        \sum_{l=1}^\infty \frac{(-1)^l}{l!} \, z^l \, s_{{+},0,l}
    \biggr), \\
    \nonumber
    \sigma (x) & =
    \prod_{j=1}^g {} (x + s_{j,1,1} \, s_{j+g,1,1})\,.
\end{align}
\end{theorem}
\begin{proof}
Note that in the rank $(N-1)$ case, we have
$$ Q_{N-1,d}(E) \cong Q(E,(1,\du)) = P(E,(1,\du))\,.$$
Using the wall-crossing formula from Proposition \eqref{proposition wall-crossing of pairs},
we have 
$$ [P(E,1,\du)]_\inv = - \left[ \mathsf e^{((0,0),E^\vee)}, \sfM_{(1,\du)} \right] \,.$$
We use the linear isomorphism from \cite[Equation 2.106]{Bu23},
$$ \xi : H_\bullet(\M_\alpha^\pl) \cong \ex^{\alpha} \cdot \mb C\left[s_{j,k,l} : \substack{(j,k)\in J,\\ l>k/2} \right]/\im T  \longrightarrow \ex^{\alpha} \cdot 
\mb C\left[s_{j,k,l} : \substack{(j,k)\in J,\ l>k/2, \\ (j,k,l) \neq (1,0,1)} \right] $$
given by 
$${\xi}(\ex^{\alpha} \cdot p + \im T) = \ex^{\alpha} \cdot \left[ \sum_{i\geqslant 0} \frac{1}{i!(-r)^i} {T}^i \circ \left( \frac{\p}{\p s_{1,0,1}} \right)^i \right] p\,,$$
where $r = \rk(\alpha)$ and $T$ is the restriction of 
$\acute T$ to the subspace $\wh H_\bullet (\M_{\alpha})$.
From \cite[Theorem 4.2]{Bu23}, we have the fundamental class
$$ \xi( \sfM_{(1,\du)} ) = \ex^{(1,\du)} \cdot \sigma(-s_{1,2,2})\,.$$
As $\sigma(-s_{1,2,2})$ contains only the classes $s_{j,1,1}$ and $s_{1,2,2}$, and $\xi$ has no effect on these classes, we can write
$$\sfM_{(1,\du)}  =  \ex^{(1,\du)} \cdot \sigma(-s_{1,2,2}) + \im T 
\qquad \text{ in } H_\bullet(\M_{(1,\du)})/\im T
\,.$$
Using the inclusion $\iota^\pl : \M^\pl \to \Mdash^\pl$,
we can write
$$\sfM_{(1,\du)}  =  \ex^{((1,\du),0)} \cdot \sigma(-s_{1,2,2}) + \im \acute{T}
\qquad \text{ in } H_\bullet(\Mdash_{(1,\du),0})/\im \acute{T}
\,.$$
Now we compute the bracket $\left[ \mathsf e^{((0,0),E^\vee)}, \sfM_{(1,\du)} \right] $.
We recall that the bracket is defined using the operator $\acute{Y}$ of the vertex algebra $\acute{V_\bullet}$ (Theorem \ref{theorem vertex algebra structure}) as
$$-\left[ \mathsf e^{((0,0),E^\vee)}, \sfM_{(1,\du)} \right] 
= \res_{z=0} \left(\acute{Y} \left(\mathsf e^{((0,0),E^\vee)}, -z\right) \left(\sfM_{(1,\du)}\right) \right) \,. $$
We use the description of $\acute{Y}$ given in \eqref{equation description of Y}.
In our case, we have $\alpha=(0,0)$, $e=1$, $p=1$ and $\alpha'=(1,\du)$, $e'=0$, $p'= \sigma(-s_{1,2,2})$.
And we have 
$$\acute{\chi}_{\sym}((\alpha,e),(\alpha',e')) =
\acute{\chi}_{\sym}(((0,0),1),((1,\du),0)) = 
-\chi(E)-\rk(E)\du = -(\nu+N\du) \,.$$
Substituting these values, we get the following
\begin{align*}
        & \acute{Y} (\mathsf e^{((0,0),E^\vee)}, -z) \left(\sfM_{(1,\du)}\right) \\
        & = {}
        \ex^{((1,\du),1)} \cdot \frac{1}{z^{\nu+N\du}}
        \exp \biggl[ -z \biggl(s_{{+},0,1} +
            \sum_{l=1}^{\infty} s_{{+},0,l+1} \,
            \frac{\p}{\p s_{+,0,l}}
        \biggr) \biggr] 
        \, \exp \biggl( -\frac{N}{z} \, \frac{\p}{\p s_{1,2,2}} \biggr) 
        \sigma (-s_{1,2,2}) \\
        & = {}
        \ex^{((1,\du),1)} \cdot \frac{1}{z^{\nu+N\du}} \
        \rho (z) \ \sigma \Bigl( \frac{N}{z} - s_{1,2,2} \Bigr),
    \end{align*}
    where the following two identities are used.
    For any polynomial $f$ not involving the variables $s_{{+},0,l}$,
    \begin{equation*}
        \exp \biggl[ -z \biggl(
            s_{{+},0,1} +
            \sum_{l=1}^{\infty} s_{{+},0,l+1} \,
            \frac{\p}{\p s_{{+},0,l}}
        \biggr) \biggr] \ f =
        \rho (z) \, f
    \end{equation*}
    and for formal variables $a, x$ of degree $2$,
    \begin{equation*}
        \exp \Bigl( a \, \frac{\p}{\p x} \Bigr)
        \, \sigma (x) = \sigma (a + x)
    \end{equation*}
\end{proof}

\section{Virasoro Constraints on Quot scheme}\label{section virasoro constraints}
    In this section, we will formulate and prove Virasoro constraints for Quot schemes using the Joyce invariant and the wall-crossing formula.
    We start with defining descendent integrals and Virasoro constraints.
    We refer the reader to \cite[Section 2.1]{BLM24} for the definition of supercommutative algebras.
\begin{definition}[Descendent Algebra]\label{definition descendent algebra}
Let $\CH^C$ denote the infinite dimensional vector space over $\C$  generated by symbols called \textit{holomorphic descendents} of the form
$$\chh_l(\gamma)\quad\textup{ for }\quad l\geqslant0,\, \gamma\in H^\bullet(C)$$
subject to the linearity relations
$$\chh_l(\lambda_1\gamma_1+\lambda_2\gamma_2)=\lambda_1\chh_l(\gamma_1)+\lambda_2\chh_l(\gamma_2)$$
for $\lambda_1, \lambda_2\in \C$. 
We define the \textit{cohomological} $\Z$-grading on $\CH^C$ by 
\begin{equation}\label{equation degree chigamma}
\deg\chh_l(\gamma) = 2l-p+q \quad \text{ for }\quad \gamma\in H^{p,q}(C)\,.
\end{equation}
Finally, we define $\D^C$ be the $\Z$-graded \textit{algebra of holomorphic descendents}
\[\D^C = \SSym\llbracket\text{CH}^C\rrbracket\,,\]
which is the completion of the supercommutative algebra generated by $\chh_l(\gamma)$. 
\end{definition}

    For any $(r,d) \in K(C)$ with $0<r<\rk(E)$, consider the Quot scheme $\quot_{r,d}(E)$. 
    Let $p_C : C \times \quot_{r,d}(E) \to C$
    and $p_Q : C \times \quot_{r,d}(E) \to \quot_{r,d}(E)$ 
    denote the projection maps.
    For any sheaf $\G$ on $C \times \quot_{r,d}(E)$,
    the geometric realization map with respect to $\G$ is 
    defined to be the algebra homomorphism
    $$ \xi_{\G} : \D^C \to H^\bullet (\quot_{r,d}(E))\,,$$
    defined on the generators as 
    $$ \xi_\G(\chh_l(\gamma)) = {p_Q}_*(\ch_{l+1-p}(\G)\,(p_C^*\gamma)) \quad \quad \text{for } \gamma \in H^{p,q}(C)\,.$$
    Following \cite{BLM24}, we have the shift in the index of the Chern character using the Hodge degree of $C$.
    This will be useful for a cleaner formulation of the Virasoro operators.
    The class $\xi_\G(\chh_l(\gamma))$ can be thought of as a class in $H^{l,l-p+q,}(\quot_{r,d}(E))$.
    Though the scheme $\quot_{r,d}(E)$ might be singular 
    and a Hodge decomposition may not exist.
    Let $\K$ denote the universal kernel on $C \times \quot_{r,d}(E)$.
    We want to study the tautological descendent integrals 
    $$ \int_{[\quot_{r,d}(E)]^\vir} \xi_\kv(D) \quad \textnormal{for} \quad  D \in \D^C\, ,$$
    which are enumerative invariants 
    of the Quot scheme.

\begin{definition}[Virasoro operators]\label{definition virasoro operators}
    We define the Virasoro operator
    $$ L_m : \D^C \to \D^C \quad \text{ for } m\geqslant 0\,, $$
    as a sum of operators $R_m$ and $T_m$, where
    \begin{itemize}
        \item $R_m$ is a derivation defined on the generators by
        \begin{equation*}
            R_m(\chh_l(\gamma)) := \left( \prod_{i=0}^m (i+l) \right) \chh_{l+m}(\gamma)\,,
        \end{equation*} 
        \item $T_m$ is multiplication by the following element of $\D^C$,
        \begin{equation*}
            T_m := \left( \sum_{i+j=m} i!j!(1-g)\chh_i(\eta)\chh_j(\eta) \right) - m!\chh_m \left(\ch(E)\,\td(C) \right)\,.
        \end{equation*}
    \end{itemize}
\end{definition}

\begin{remark}
    As observed in \cite{BLM24}, the linear part of the Virasoro operators, i.e., $T_m$ is closely related to the obstruction theory of the moduli space we consider.
    The similarity between obstruction theory and Virasoro constraints below is a general phenomenon which can be used to guess the correct formulation of the constraints.
    Recall that the deformation-obstruction theory of the Quot scheme can be seen as
    $$T^\vir \quot_{r,d}(E) =  R{(p_{Q})}_* R\HOM
    \big([p_C^*E^\vee \to \K^\vee], \K^\vee\big) \,. $$
    The Chern character of the virtual tangent bundle can be expressed in terms of descendents using Grothendieck-Riemann-Roch.
    Let $\Delta: C \to C \times C$ denote the diagonal embedding
    and let 
    $$\Delta_*(\td(C))= \sum_t \gamma_t^R \otimes \gamma_t^R $$ 
    be the Kunneth decomposition such that $\gamma_t^L \in H^{p_t^L,q_t^L}(C)$ for some $p_t^L,q_t^L$. 
    Then
    \begin{align*}
    \ch(T^\vir \quot_{r,d}(E)) = - \xi_{\K^\vee}
    \left( \sum_{i,j,t} (-1)^{i+1-p_t^L}\chh_i(\gamma_t^L) \, \chh_j(\gamma_t^R)
    - \sum_{m} \chh_m(\ch(E) \cup \td(C)) \right) \,.
    \end{align*}
    We can define $T_m$ as
    $$\left( \sum_{i+j=m} \sum_t (-1)^{1-p_t^L} \, i!j! \chh_i(\gamma_t^L) \, \chh_j(\gamma_t^R) \right)
    - m!\,\chh_m(\ch(E) \cup \td(C))\,.$$
    It is easy to see that this expression simplifies to the expression given in the definition \ref{definition virasoro operators}.
\end{remark}
    As noted in \cite{BLM24}, the operators 
    $\{L_m\}_{m\geqslant-1}$ satisfy the Virasoro bracket relations
    $$ [L_m,L_n] = (m-n) L_{m+n} \,, $$
    where the bracket denotes the usual commutator of operators
    in $\End(\D^C)$.
    
    \begin{definition}[Virasoro constraints]\label{Conjecture Virasoro quot}
        The Quot scheme $\quot_{r,d}(E)$ is said to satisfy
        the Virasoro constraints if we have
        $$\int_{[\quot_{r,d}(E)]^\vir} \xi_{\kv} (L_m(D)) = 0 \quad \text{ for any } m\geqslant 0, D \in \D^C\,.$$
    \end{definition}
We will formulate the Pair Virasoro constraints which will be used to prove the Virasoro constraints for Quot schemes.
Let \[\D^\pa=\D^C\otimes\D^C\]
be the algebra of pair descendents. We denote the generators of the first copy of $\D^C$ by  $\chhE_l(\gamma)$ and 
the generators of the second copy by  $\chhF_l(\gamma)$. 
Given the universal kernel $\K$ on $C\times \quot_{r,d}(E)$, we have a geometric realization map
\[\xi_{(E,\kv)}: \D^\pa\to H^\bullet(\quot_{r,d}(E))\,,\]
defined by
\begin{align*}
\xi_{(E,\kv)}\left(\chhE_l(\gamma)\right)&=\xi_{p_C^* E^\vee}\left(\chh_l(\gamma)\right)=\begin{cases}\int_C \gamma\cdot \ch(E^\vee)&\textup{ if }l=0\,,\\
0 & \textup{ otherwise}\,,\end{cases}\\
\xi_{(E,\kv)}\left(\chhF_l(\gamma)\right)&=\xi_{\kv}\left(\chh_l(\gamma)\right).
\end{align*}
This geometric realization map factors through $\xi_\kv$, i.e.
we have the commutative diagram,
\begin{center}
\begin{tikzcd}
\D^\pa\arrow[d, swap, "\upsilon_E"] \arrow[rd,"\xi_{(E,\kv)}"]& \\
\D^C\arrow[r, "\xi_\kv"]& H^\bullet(\quot_{r,d}(E))
\end{tikzcd}
\end{center}
where $\upsilon_E$ is defined on the generators as
$$\upsilon_E(\chhF_l(\gamma)) = \chh_l(\gamma) \quad \textup{ and }
\quad 
\upsilon_E(\chhE_l(\gamma)) = \begin{cases}\int_C \gamma\cdot \ch(E^\vee)&\textup{ if }l=0\,,\\
0 & \textup{ otherwise}\,.\end{cases}$$
We define the pair Virasoro operator $L_m^{\pa}\colon \D^\pa\to \D^\pa$, for $m\geqslant 0$, as the sum $R_m^{\pa}+T_m^{\pa}$ where
\begin{itemize}
\item $R_m^{\pa}$ is a derivation defined on generators as
\[R_m^{\pa}=R_m\otimes \id+\id\otimes R_m\,,\]
\item $T_m^{\pa}$ is multiplication by an element $T_m^{\pa}$, which is defined as follows.
Let $\Delta: C \to C \times C$ denote the diagonal embedding
and let $$ \sum_t \gamma_t^R \otimes \gamma_t^R = \Delta_*(\td(C))$$
be the Kunneth decomposition such that $\gamma_t^L \in H^{p_t^L,q_t^L}(C)$ for some $p_t^L,q_t^L$. 
Define $$T_m^{\pa}:=\sum_{i+j=m} i!j!
\sum_t (-1)^{1- p_t^L}(\chhF_i - \chhE_i)(\gamma_t^L) \cdot \chhF_j(\gamma_t^R)\in \D^\pa \,.$$
\end{itemize}
The operators $L_m$ and $L_m^\pa$ are related by the following commutative diagram
\begin{equation}\label{diagram virasoro and pair}
\begin{tikzcd}
\D^\pa\arrow[d, swap, "\upsilon_E"] \arrow[r,"L_m^\pa"]& \D^\pa \arrow[d, "\upsilon_E"]\\
\D^C\arrow[r, "L_m"]& \D^C \,.
\end{tikzcd}
\end{equation}
To check the commutativity of the diagram, it is enough to see that
$$ \upsilon_E(T_m^\pa) = T_m$$
which follows from the identity
\begin{align*}
\sum_{t}(-1)^{1-p_t^L}\left(\int_C \gamma_t^L \cdot \ch(E^\vee) \right)\chh_m(\gamma_t^R)&=\sum_t \left(\int_C \gamma_t^L \cdot \ch(E) \right)\chh_m(\gamma_t^R)
=\chh_m(\ch(E) \td(C)).
\end{align*}
We can define the pair version of Virasoro constraints similarly.
The Quot scheme $\quot_{r,d}(E)$ is said to satisfy
        the pair Virasoro constraints if we have
        \begin{equation}\label{Conjecture Virasoro pair}
            \int_{[\quot_{r,d}(E)]^\vir} \xi_{(E,\kv)} (L^\pa_m(D)) = 0 \quad \text{ for any } m \geqslant 0, D \in \D^\pa\,.
        \end{equation}
Due to the commutative diagram \eqref{diagram virasoro and pair}, it follows that the Virasoro constraints on the Quot scheme
are implied by the pair Virasoro constraints.
So it will be sufficient to prove the pair Virasoro constraints for Quot schemes.
    
    Now we will reformulate the pair Virasoro constraints in the language of vertex algebra structure, so that we can use the Joyce's invariants and wall-crossing formula.
    To prove the pair Virasoro constraints, it is suitable to work with a larger vertex algebra $V^\pa_{\bullet}$, which is constructed from the homology of the stack parametrizing pair of complexes.
    Recall the (higher) moduli stack $\M$ of perfect complexes on $C$.
    There is a universal perfect complex $\W$ on $C \times \M$.
    Let 
    $$\P:= \M \times \M$$
    be the (higher) moduli stack parametrizing pair of perfect complexes on $C$
    and let $\E$ and $\mb F$ denote the pullbacks of the universal complex $\W$ from the first and second factor respectively on $C \times \P$.
    As in section \ref{section Vertex Algebra Structure and Invariants}, we will define the pair vertex algebra from homology of the stack $\P$.
    For this, we define the complex $\EXT^\pa$ on $\P \times \P$ as the following.
    Let $\pi_{ij}$ denote the projection from 
    $C \times \P \times \P$ to the $i$-th and $j$-th factors
    and let $q_2,q_3$ denote the projections from $\P \times \P$ to its factors. 
    Following \cite[Definition 4.4]{BLM24}, 
    we define $\EXT^\pa$ to be the perfect complex 
    \begin{align}\label{equation definition ext^pa}
        \EXT^\pa := (\pi_{23})_*\biggl(\pi_{12}^*(\F \oplus \E[1])^\vee \otimes \pi_{13}^*\F \biggr)
    \end{align}
    and 
    $$ \mc E ^\pa := (\EXT^\pa)^\vee$$
    on $\P \times \P$. 
    For any $\alpha^\pa = (\alpha_1,\alpha_2) \in K(C)^{\oplus 2}$,
    let $\P_{\alpha^\pa}$ denote the connected component $\M_{\alpha_1} \times \M_{\alpha_2}$.
    The stack $\P$ can be decomposed as 
    $$ \P= \bigsqcup_{\alpha^\pa \in K(C)^{\oplus 2}} \P_{\alpha^\pa} \,.$$
    For $\alpha^\pa, \beta^\pa \in K(C)^{\oplus 2}$,
    we define the pairing 
    $$\chi^\pa(\alpha^\pa,\beta^\pa) = \rank(\EXT^\pa_{\alpha^\pa,\beta^\pa})\,.$$
    For any $\alpha^\pa \in K(C)^{\oplus 2}$, 
    define the graded vector space of shifted homology
    $$ V^\pa_{\bullet, \alpha^\pa} := 
    H_{\bullet-2 \chi^\pa(\alpha^\pa,\alpha^\pa)}(\P_{\alpha^\pa})\,.$$
    The pair vertex algebra has the underlying graded vector space $V_\bullet^\pa$, which is defined as
    $$ V_\bullet^\pa := \bigoplus_{\alpha^\pa} V^\pa_{\bullet, \alpha^\pa}\,.$$
    The vertex algebra structure on $V_\bullet^\pa$
    is constructed exactly as in Theorem \ref{theorem vertex algebra structure}, with $\Mdash$ replaced by $\P$, $\acute{\mc E}$ by $\mc E^\pa$ and $\acute{\chi}$ by $\chi^\pa$.

Let $\alpha \in K(C)$.
There is an algebra homomorphism
$\xi_\alpha : \D^C \to H^\bullet(\M_\alpha)$
defined by
$$ \xi_\alpha(\chh_l(\gamma)) = 
\begin{cases}
    \ch_{l+1-p}(\W_\alpha) \setminus \gamma^\vee \quad & 
    \text{ if } \deg \chh_l(\gamma)>0\,,\\
    \int_{C} \gamma \cdot \ch(\alpha) & 
    \text{ if } \deg \chh_l(\gamma)=0\,, \\
    0 &\text{ otherwise}\,,
\end{cases}$$
where $\gamma^\vee \in H_\bullet(C)$ is the the dual of $\gamma \in H^{p,q}(C)$.
Recall the classes $S^\alpha_{j,k,l} \in H^\bullet(\M_\alpha)$ defined in \eqref{equation definition S_j,k,l classes}.
Clearly, these classes are realized by $\xi_\alpha$.
Moreover, \cite[Theorem 4.15]{gros20} says 
that $H^\bullet(\M_\alpha)$ is generated by the classes
$S^\alpha_{j,k,l}$ as a graded algebra.
This implies that $\xi_\alpha$ is surjective.
We do the same thing for the pair version.
For $\alpha^\pa = (\alpha_1,\alpha_2)$, we define
the algebra homomorphism
$$\xi_{\alpha^\pa} := \xi_{\alpha_1} \otimes \xi_{\alpha_2}: 
\D^\pa \longrightarrow H^\bullet(\P_{\alpha^\pa})\,.$$
The map $\xi_{\alpha^\pa}$ is surjective.
Also note that $\ker(\xi_\alpha)$ is generated by the symbols of the form
$\chh_0(\gamma)$.
As $R_m(\chh_0(\gamma))=0$,
it follows that $\xi_{\alpha^\pa}(L^\pa_m(\ker(\xi_{\alpha^\pa})))=0$.
So the Virasoro operators $L^\pa_m$ factor through the quotient 
$H^\bullet(\P_{\alpha^\pa})$.
In other words, we have operators $\mL^\pa_m$ on the cohomology ring $H^\bullet(\P_{\alpha^\pa})$ satisfying the following commutative diagram.
$$
    \begin{tikzcd}
    \D^\pa \arrow[r,"L^\pa_m"] \arrow[d,"\xi_{\alpha^\pa}"] & \D^\pa \arrow[d,"\xi_{\alpha^\pa}"]\\   
    H^\bullet(\P_{\alpha^\pa}) \arrow[r,"\mL^\pa_m"]& H^\bullet(\P_{\alpha^\pa}) \,.
\end{tikzcd}
$$
\begin{theorem}\label{theorem virasoro conformal}
     For each $m \geqslant 0$, there exist an operator $\L_m^\pa $ on the shifted homology $V^\pa_\bullet$ such that 
    the operators $\mL_m^\pa$ and $\L_m^\pa$ are dual to each other with respect to the perfect pairing 
    $$H^\bullet(\P_{\alpha^\pa}) \otimes V^\pa_{\bullet,\alpha^\pa} \to \C \,.$$
\end{theorem}
\begin{proof}
This is mainly proved in \cite[Section 4.3]{BLM24},
using a lattice vertex algebra structure on $V_\bullet^\pa$.
All the objects here are exactly the same as in \cite[Section 4.3]{BLM24}.
So we will omit the details of the construction of the lattice vertex algebra structure and only state the results.
It is shown in \cite[Theorem 4.7]{BLM24} that the vertex algebra
$V^\pa_\bullet$ is isomorphic to a lattice vertex algebra
$\C[\Lambda^\pa_{\text{sst}}] \otimes \D_{\Lambda^\pa}$.
In \cite[Lemma 4.8]{BLM24}, it is shown that
$H^\bullet(\P_{\alpha^\pa})$ is isomorphic to a graded algebra $\D^{\pa}_{\alpha^\pa}$.
An abstract cap product 
$$ \cap : \D^\pa_{\alpha^\pa} \times \D_{\Lambda^\pa} \to \D_{\Lambda^\pa}$$
can be constructed as in \cite[Equation (40)]{BLM24}.
Moreover, the commutative diagram \cite[Diagram (41)]{BLM24}
implies that this abstract cap product coincides with the topological cap product 
$$ \cap : H^\bullet(\P_{\alpha^\pa}) \times H_\bullet(\P_{\alpha^\pa}) \to H_\bullet(\P_{\alpha^\pa})\,,$$
via the isomorphisms mentioned above.
Using the lattice vertex algebra structure, \cite[Section 4.3]{BLM24} defines a conformal element $\omega$ in $V^\pa_\bullet$,
which gives rise to the vertex Virasoro operators $\L_m^\pa : V^\pa_\bullet \to V^\pa _\bullet$ for $m \geqslant 0$.
Finally, \cite[Theorem 4.12]{BLM24} shows that the operators 
$\mL_m^\pa$ and $\L_m^\pa$ are dual to each other with respect to the perfect pairing coming from the topological cap product.
\end{proof}
We define the space of primary states in the pair vertex algebra
$V_\bullet^\pa$ as
$$ \PS := \{a \in V_\bullet^\pa : \L_m^\pa(a) = 0 \text{ for all } m \geqslant 0\}\,.$$

Let us consider the Quot scheme $\quot_{r,d}(E)$.
Fix $\alpha = (\rk(E)-r,d-\deg(E))$ and $\alpha_{\pa} = (\alpha(E),\alpha)$.
The sheaves $p_C^*E^\vee$ and $\kv$ on $C \times \quot_{r,d}(E)$ induce a map
$$\mf g : \quot_{r,d}(E) \to \P_{\alpha^\pa}$$
such that $(\id_C \times \mf g)^*\E = p_C^*E^\vee$ 
and $(\id_C \times \mf g)^*\F = \kv$.
This map sends a closed point $[\vp: E \to Q]$ to the pair $(E^\vee,(\ker\vp)^\vee)$.
Let us consider the pushforward class
$$\mathsf{Quot}_{r,d}(E) :={\mathfrak g}_*([\quot_{r,d}(E)]^\vir)  \in H_\bullet (\P_{\alpha^\pa})\,.$$

\begin{lemma}\label{Lemma virasoro primary states}
    The Virasoro constraints, i.e., the condition \eqref{Conjecture Virasoro quot} is equivalent to the following condition
    \begin{equation*}\label{condition virasoro primary states}
        \mathsf{Quot}_{r,d}(E) \in \PS\,. 
    \end{equation*}
\end{lemma}
\begin{proof}
    We have already noticed that, to prove the Virasoro constraints \eqref{Conjecture Virasoro quot},
    it is sufficient to prove the pair Virasoro constraints \eqref{Conjecture Virasoro pair}.
    The map $\mf g$ induces a morphism 
$${\mf g}^* : H^{\bullet}(\P_{\alpha^\pa}) \longrightarrow 
H^{\bullet}(\quot_{r,d}(E))\,.$$
It is easy to see that ${\mf g}^*$ makes the following diagram
commutative,
    \begin{equation*}
    \begin{tikzcd}
    \D^\pa\arrow[rd, swap, "\xi_{(E,\kv)}"] \arrow[r,"\xi_{\alpha^{\pa}}"]& H^{\bullet}(\P_{\alpha^\pa}) \arrow[d, "{\mf g}^*"]\\
    & H^\bullet(\quot_{r,d}(E)) \,.
    \end{tikzcd}
\end{equation*}
Since the map $\xi_{\alpha^\pa}$ is surjective,
the pair Virasoro constraints \eqref{Conjecture Virasoro pair} is equivalent to the following vanishing,
$$\int_{[\quot_{r,d}(E)]^\vir} \mf g^* (\mL_m^\pa(D)) = 0 \quad \text{ for any } m \geqslant 0, D \in H^{\bullet} (\P_{\alpha^\pa})\,.$$
For any $m \geqslant 0$ and $D \in H^{\bullet}(\P_{\alpha^\pa})$,
we have
\begin{align*}
\int_{[\quot_{r,d}(E)]^\vir} {\mf g}^* (\mL^\pa_m(D))
& = \int_{\mathsf{Quot}_{r,d}(E)} \mL^\pa_m(D) \\
& = \int_{\L^\pa_m \left(\mathsf{Quot}_{r,d}(E)\right)} D \,,
\end{align*}
where the second equality is due to Theorem \ref{theorem virasoro conformal}.
This integral vanishes for all $D \in H^{\bullet}(\P_{\alpha^\pa})$ if and only if $\L^\pa_m(\mathsf{Quot}_{r,d}(E)) = 0$.
This is equivalent to saying that $\mathsf{Quot}_{r,d}(E)$
is a primary state. 
\end{proof}
Recall the moduli stack $\Mdash$ parametrizing $E^\vee$-pairs 
and the universal pair $\Theta : E^\vee \boxtimes \V \to \U$ on $C \times \Mdash$.
The universal pair induces a map 
$$ \Xi :\Mdash \to \P\,.$$
and a map for the projective linear version,
$$ \Xi^\pl :\Mdash^\pl \to \P^\pl\,.$$
Consider the commutative diagram
$$
    \begin{tikzcd}
    \Mdash \arrow[r,"\Xi"] \arrow[d,"\Pidash^\pl"] & \P \arrow[d,"\Pi^\pl_\pa"]\\   
    \Mdash^\pl \arrow[r,"\Xi^\pl"]& \P^\pl \,.
\end{tikzcd}
$$
Consequently, we will have a commutative diagram on the shifted homologies
$$
\begin{tikzcd}
    \acute V_\bullet \arrow[r,"{\Xi}_*"] \arrow[d,"\Pidash^\pl_*"] & V^\pa_\bullet \arrow[d,"(\Pi^\pl_\pa)_*"]\\   
    \wh{H}_\bullet(\Mdash^\pl) \arrow[r,"\Xi^\pl_*"]& \wh{H}_\bullet(\P^\pl) \,.
\end{tikzcd}
$$
The vertex algebra structure on $V_\bullet^\pa$ induces 
a graded Lie algebra structure (due to \cite{Bor86})
on 
$$V^\pa_{\bullet}/\im (T^\pa) \cong \wh{H}_\bullet(\P^\pl)\,,$$
where the isomorphism is due to \cite[Theorem 4.8]{Joy21}.
Let us define,
$$\widecheck V^\pa_\bullet := V^\pa_{\bullet}/\im (T^\pa)$$
and let $\widecheck{\PS}$ be the image of $\PS$ in $\widecheck{V}^\pa_\bullet$, i.e.,
$$\widecheck{\PS}:= \PS/T^\pa(\PS)\,.$$
    Recall the wall-crossing formula 
    \eqref{equation wall crossing for quot to e^001} 
    for the Quot scheme $\quot_{r,d}(E)$  takes place in 
    $\wh{H}_\bullet(\Mdash_{\alpha,1}^\pl)$.
    As the primary states are defined in $V^\pa_\bullet$, 
    we would like to make sense of the wall-crossing in $V^\pa_\bullet$.
    For this, we will use a partial lift of the Lie bracket to $V^\pa_\bullet$, which is already introduced in \cite{BLM24}.
    \begin{proposition}\label{proposition lifting bracket}
        There is a well-defined linear map 
        $$ \ad: \widecheck V^\pa_\bullet \times V^\pa_\bullet \to V^\pa_\bullet$$
        which satisfies that $\overline{\ad(\bar{a},b)}= [\bar{a},\bar{b}]$ for any $a,b \in V^\pa_\bullet$.
        Here $\bar{x}$ denote 
        the image of $x \in V^\pa_\bullet$ in
        $V^\pa_\bullet/\im(T^\pa) = \widecheck V^\pa_\bullet$.
        Moreover, $\PS$ is a subrepresentation of $V^\pa_\bullet$ with respect to the Lie  subalgebra $\widecheck{\PS}$, i.e. the map $\ad$ restricts to
        $$ \ad : \widecheck{\PS} \times \PS \to \PS\,.$$
    \end{proposition}
    \begin{proof}
    The definition of the map $\ad$ follows from \cite[Lemma 3.12]{BLM24}.
    The fact that $\PS$ is a subrepresentation with respect to $\widecheck{\PS}$ follows from \cite[Proposition 3.13]{BLM24}, taking $i=0$.
    \end{proof}  
    We use another lemma from \cite{BLM24}.
    Recall that $\eta \in H^2(C)$ is a top form.
    For notational convenience, 
    let us define, for any $v \in V^\pa_\bullet$,
    $$\chhE_1(\eta) \cap v := \xi_{\alpha^\pa}(\chhE_1(\eta)) \cap v
    \quad \text{ if } v \in V^\pa_{\bullet, \alpha^\pa}\,.$$ 
    \begin{lemma}[Lemma 5.11, \cite{BLM24}]\label{Lemma 5.11 of BLM}
        \begin{enumerate}
            \item Let $\bar{u} \in \widecheck{V}_\bullet \subset \widecheck{V}^\pa_\bullet$ and 
            $v \in V^\pa_\bullet$ such that 
            $\chhE_1(\eta) \cap v = 0$, then
            $$\chhE_1(\eta) \cap \ad(\bar{u},v) = 0,.$$
            \item Let $u,v \in V^\pa_{\bullet, (\alpha_1,\alpha_2)}$, 
            with $\rk(\alpha_1) > 0$, be such that 
            $$\bar{u}= \bar{v} \, \text{ in } \widecheck{V}^\pa_\bullet \qquad
            \text{ and } \qquad
            \chhE_1(\eta) \cap u = 0 = \chhE_1(\eta) \cap v\,,$$
            then $u = v$ in $V^\pa_\bullet$.
        \end{enumerate}
    \end{lemma}
    
    \begin{theorem}\label{Theorem Virasoro constraints}
        Let $0<r<\rk(E)$.
        Then the Quot scheme $\quot_{r,d}(E)$ satisfies the Virasoro constraints \eqref{Conjecture Virasoro quot}, i.e.,
        $$\int_{[\quot_{r,d}(E)]^\vir} \xi_{\kv} (L_m(D)) = 0 \quad \text{ for any } m\geqslant 0, D \in \D^C\,.$$
    \end{theorem}
    \begin{proof}
    Using Lemma \ref{Lemma virasoro primary states}, 
    we only need to prove that the class $\mathsf{Quot}_{r,d}(E) \in V^\pa_\bullet$ is a primary state.
    Let $\alpha = (\rk(E)-r, d-\deg(E))$.
    Recall the wall-crossing formula \eqref{equation wall crossing for quot to e^001},
    \begin{align*}
        [\quot_{r,d}(E)]_\inv
            =\sum_{\substack{n\geqslant 1, \sum \alpha_i = \alpha, \\ 
            \alpha_i=(r_i,d_i),\ r_i>0 \ \forall i}}
            U'(\alpha_1,\dots,\alpha_n) \cdot
            & \left[\left[ \cdots
     \left[ \left[\mathsf e^{((0,0),E^\vee)},\sfM_{\alpha_1} \right], \sfM_{\alpha_2} \right],\dots \right], \sfM_{\alpha_n} \right]
    \end{align*}
    for some coefficients $U'(\alpha_1,\dots,\alpha_n)$ and the sum is finite.
    This formula takes place in $\wh{H}_\bullet(\Mdash_{\alpha,1}^\pl)$.
    We push this equation forward
    along $\Xi^\pl$ to get an equality in $\wh{H}_\bullet(\P^\pl)
    \cong \widecheck V^\pa_\bullet$.
    However, we have to make sure that $\Xi^\pl$ respects the bracket operations of the two Lie algebras.
    To see this, we recall the complexes $\acute\EXT$ on $\Mdash \times \Mdash$ (defined in \eqref{equation definition of ext complex}) and $\EXT^\pa$ on $\P \times \P$ (defined in \eqref{equation definition ext^pa}), which were used to define the bracket operations.
    The restrictions of $\acute\EXT$ and $(\Xi \times \Xi)^*(\EXT^\pa)$ on $\Mdash \times \M \subset \Mdash \times \Mdash$ coincide as $K$-theory classes.
    Consequently, if $u \in \acute V_\bullet$ and $v \in V_\bullet \subset \acute V_\bullet$,
    then the state-field correspondences in both the algebras coincide, i.e.,
    $$ \Xi_*\left(\acute Y(u,z)v\right) = Y^\pa\left((\Xi_*u),z\right)(\Xi_*v)\,.$$
    It follows that the bracket operations appearing in the wall-crossing formula remain compatible when we pushforward along $\Xi^\pa$.
    Thus we have the wall-crossing formula \eqref{equation wall crossing for quot to e^001} in $\widecheck V^\pa_\bullet$ as well.
    Changing the coefficients suitably, we write the wall-crossing formula in $\widecheck V^\pa_\bullet$ as the following,
    \begin{align}\label{equation wall crossing quot in Vpa-check}
    \mathsf{Quot}_{r,d}(E)
            =\sum_{\substack{n\geqslant 1, \sum \alpha_i = \alpha, \\ 
            \alpha_i=(r_i,d_i),\ r_i>0 \ \forall i}}
            U''(\alpha_1,\dots,\alpha_n) \cdot
            & \left[ \sfM_{\alpha_n},
            \left[ \cdots, \left[ \sfM_{\alpha_2}, \left[\sfM_{\alpha_1}, \mathsf e^{((0,0),E^\vee)} \right] \right], \cdots \right] \right]
    \end{align}
    for some coefficients $U''(\alpha_1,\dots,\alpha_n)$.
    Now we lift this formula to $V^\pa_\bullet$.
    First we note that 
    $$\chhE_1(\eta) \cap \mathsf e^{((0,0),E^\vee)}=0\,.$$
    Using part (1) of Lemma \ref{Lemma 5.11 of BLM}, we have,
    for any $\alpha_1 \in K(C)$,
    $$\chhE_1(\eta) \cap \ad \left( \sfM_{\alpha_1}, \mathsf e^{((0,0),E^\vee)} \right)=0\,.$$
    Similarly using part (1) of Lemma \ref{Lemma 5.11 of BLM} repeatedly, it follows that,
    for $\alpha_1,\dots,\alpha_n \in K(C)$,
    $$\chhE_1(\eta) \cap \ad \left( \sfM_{\alpha_n},
            \ad \left( \cdots, \ad \left( \sfM_{\alpha_2}, \ad \left(\sfM_{\alpha_1}, \mathsf e^{((0,0),E^\vee)} \right) \right), \cdots \right) \right)=0\,.$$
    On the other hand, the realization of the class $\chhE_1(\eta)$ on the Quot scheme $\quot_{r,d}(E)$ is 
    $$\xi_{p_C^*E}(\chhE_1(\eta)) = {p_Q}_*(\ch_1(p_C^*E \cup p_C^*\eta) = 0\,.$$
    This shows that $$ \chhE_1(\eta) \cap [\quot_{r,d}(E)]_\inv=0\,.$$
    Finally, part (2) of Lemma \ref{Lemma 5.11 of BLM}
    says that we can lift the equality \eqref{equation wall crossing quot in Vpa-check} to $V^\pa_\bullet$,
    i.e., we get the following equality in $V^\pa_\bullet$,
    \begin{align}\label{equation wall crossing in Vpa}
    \nonumber
    \mathsf{Quot}_{r,d}(E)
            =\sum_{\substack{n\geqslant 1, \sum \alpha_i = \alpha, \\ 
            \alpha_i=(r_i,d_i),\ r_i>0 \ \forall i}}
            &U''(\alpha_1,\dots,\alpha_n) \cdot \\
            &\ad \left( \sfM_{\alpha_n},
            \ad \left( \cdots, \ad \left( \sfM_{\alpha_2}, \ad \left(\sfM_{\alpha_1}, \mathsf e^{((0,0),E^\vee)} \right) \right), \cdots \right) \right)\,.
    \end{align}
    It is proved in \cite[Theorem 5.12]{BLM24} that 
    the classes $\sfM_{\alpha_i}$ are primary states, i.e.
    $$ \sfM_{\alpha_i} \in \widecheck{\PS} 
    \qquad \text{ for any } \alpha_i \in K(C)
    \,.$$
    Clearly $\mathsf e^{((0,0),E^\vee)} \in \PS$.
    Using second part of Proposition \ref{proposition lifting bracket}, 
    we conclude that the right hand side of the wall-crossing formula \eqref{equation wall crossing in Vpa} 
    is a primary state.
    Hence $\mathsf{Quot}_{r,d}(E)$ is a  primary state.
    This proves that the Virasoro constraints are satisfied on the Quot schemes.
\end{proof}

\begin{remark}\label{remark Virasoro for pairs}
    Given $\alpha \in K(C)$,
    recall the moduli space $P(E,\alpha) := \Mdash^{0+\ss}_{\alpha,1}$
    defined in Section \ref{Section required objects}.
    We have introduced the invariant 
    $[P(E,\alpha)]_\inv$ in Theorem \ref{theorem quot invariant}
    and proved a wall-crossing formula in Proposition \ref{proposition wall-crossing of pairs}.
    We can formulate and prove the Virasoro constraints for $P(E,\alpha)$ in exactly the same way as for Quot schemes.
    There will be no change in the Virasoro operators 
    as the obstruction theory is the same (given by \eqref{equation POT on stable pairs}).
    Indeed, using the wall-crossing formula for $[P(E,\alpha)]_\inv$ in the proof of Theorem \ref{Theorem Virasoro constraints},
    we see that (an appropriate image of) the invariant $[P(E,\alpha)]_\inv$ is also a primary state.
\end{remark}

\vspace{.2cm}

\section{Virtual Intersections on Quot schemes}\label{section intersection theory}
We use Joyce's invariant and the Virasoro constraints to derive few results on virtual intersections on Quot schemes.
Let $E$ be a vector bundle on $C$ and $0<r<\rk(E)$.
We consider the Quot scheme $\quot_{r,d}(E)$ and let 
$$ 0 \to \K \xrightarrow{\Phi} p_C^*E \to \Q \to 0$$
denote the universal short exact sequence on the $C \times \quot_{r,d}(E)$.
Let $\alpha=(\rk(E)-r,d-\deg(E))$.
Recall the moduli stack $\Mdash_{\alpha,1}$ parametrizing $E^\vee$-pairs of type $(\alpha,1)$
and the universal pair
$$ \Theta: E^\vee \boxtimes \V_1 \to \U_\alpha$$
on $C \times \Mdash_{\alpha,1}$.
We have a map 
$$\mf h : \quot_{r,d}(E) \to \Mdash_{\alpha,1}$$
such that 
$$\mf h^*(E^\vee \boxtimes \V_1 \xrightarrow{\Theta} \U_{\alpha}) = \left(p_C^*E^\vee \xrightarrow{\Phi^\vee} \K^\vee \right) \,.$$
Recall the symplectic basis $\{\e_{j,k} : (j,k) \in J \}$
for the cohomology of $C$ 
and the tautological classes
$S^{\alpha,1}_{j,k,l} \in H^{2l-k}(\Mdash_{\alpha,1})$ 
for $(j,k) \in J$, $l \geqslant k/2$ (defined in \eqref{equation definition S_j,k,l classes}).
The pullbacks of these classes become tautological classes on the Quot scheme $\quot_{r,d}(E)$,
i.e.,
$$ \mf h^*S^{\alpha,1}_{j,k,l} = \ch_l(\K^\vee)\setminus \e_{j,k}^\vee\,.$$
Abusing notation, let us only write $S_{j,k,l}$ in place of
$\mf h^*(S^{\alpha,1}_{j,k,l})$.
The holomorphic descendents can be realized as follows,
\begin{align*}
    \xi_{\K^\vee}(\chh_l(\e_{1,2})) &= S_{1,0,l}\,,\\
    \xi_{\K^\vee}(\chh_l(\e_{j,1})) &=  \begin{cases}
         S_{j,1,l} \quad &\textnormal{ if } 1 \leqslant j \leqslant g \,,\\
         S_{j,1,l+1} \quad &\textnormal{ if } g+1 \leqslant j \leqslant 2g \,,
    \end{cases}\\
    \xi_{\K^\vee}(\chh_l(\e_{1,0})) &= S_{1,2,l+1}\,,
\end{align*}
So the Virasoro constraints \eqref{Conjecture Virasoro quot} can be written in terms of the tautological classes $S_{j,k,l}$'s. 
The geometric realization of the element $T_m$ is the following
\begin{equation*}
    \xi_{\K^\vee}(T_m) = \left( (1-g)\sum_{i+j=m}i!j!\ S_{1,0,i}S_{1,0,j} \right)- m!(\left(\rk(E) \cdot S_{1,2,m+1} + \chi(E) \cdot S_{1,0,m} \right)\,.
\end{equation*}
So the constraints yield the following relation, for $m \geq 0$,
\begin{equation}\label{equation virasoro relations}
    \int R_m(D) = - \int \left( (1-g) \sum_{i+j=m}i!j!\ S_{1,0,i}S_{1,0,j} - m!(\rk(E) \cdot S_{1,2,m+1} + \chi(E) \cdot S_{1,0,m}) \right) \cdot D\,.
\end{equation}
where $D$ is any polynomial in $S_{j,k,l}$ classes and
the integrals (here and throughout) are understood to be taken over the virtual fundamental class $[\quot_{r,d}(E)]^\vir$.
The operators $R_m$ act as derivations with the action on $S_{j,k,l}$ classes being 
$$R_m(S_{j,k,l})= \prod_{i=0}^m (i+l) \cdot S_{j,k,l+m}\,.$$

The relations obtained from the Virasoro constraints cannot determine all the tautological intersection numbers.
As discussed in the introduction, 
the Virasoro constraints are some universal relations which hold in more generality.
So the Virasoro constraints are not expected to give 
information for some particular moduli space.
However, the Virasoro constraints yield many useful relations among intersection numbers.
In the following, we denote the virtual (or expected)
dimension of $\quot_{r,d}(E)$ as $\dim Q$.
\begin{proposition}\label{proposition intersection of S_12 classes from S_10 classes}
    Let $P$ be any polynomial in $S_{1,0,l}$ and $S_{1,2,l}$-classes of weighted degree $2\dim Q$.
    Then there exists a polynomial $\wt P$, only in $S_{1,0,l}$-classes, such that 
    $$ \int P(S_{1,0,*},S_{1,2,*}) = \int \wt P (S_{1,0,*})\,.$$
\end{proposition}
\begin{proof}
    We need to prove the result for all the monomials $M$ of the form,
    $$M= \prod_{t=1}^{\dim Q} (S_{1,2,t+1})^{\lambda_t} \cdot A(S_{1,0,*})\,,$$
    where
    $A$ is a monomial in $S_{1,0,*}$ classes of degree $2\mu \geqslant 0$ and
    $$\sum_{t=1}^{\dim Q} t \lambda_t = \dim Q -\mu\,. $$
    For such a monomial $M$, define  
    $$\#_{1,2,*}(M) := \sum_{t=1}^{\dim Q} \lambda_t $$
    We will proceed by induction on the number $\#_{1,2,*}(M)$.
    
    For the base case, let $\#_{1,2,*}(M)=1$.
    Then $\lambda_n=1$ for some $n$ and all other $\lambda_t$ is 0.
    Precisely, $M = S_{1,2,n+1} \cdot A$.
    Putting $m=n$ and $D=A$ in Virasoro relation \eqref{equation virasoro relations}, we have
    $$\int R_n(A) = \int \left( -(1-g)\sum_{i+j=n}i!j!\ S_{1,0,i}S_{1,0,j} + n!(\left(\rk(E) \cdot S_{1,2,n+1} + \chi(E) \cdot S_{1,0,n} \right) \right)\cdot A \,.$$
    We note that $R_n(A)$ is a polynomial in $S_{1,0,*}$ classes.
    It follows that
    $$\int S_{1,2,n+1} \cdot A = \frac{1}{n! \cdot \rk(E)}
    \int \left( R_n(A) +
    (1-g)\sum_{i+j=n}i!j!\ S_{1,0,i}S_{1,0,j}A - n!\chi(E)\ S_{1,0,n} A \right) \,.
    $$
    We take $\wt M$ to be the polynomial appearing in the right-hand side.
    This proves the base case.
    
    Let us assume the proposition is true whenever $\#_{1,2,*}(M) \leqslant \Lambda$
    and we consider a monomial $M$ with $\#_{1,2,*}(M) =\Lambda+1$.
    As $\Lambda +1 >0$, there exists $n \in \{1,\dots,\dim Q\}$
    such that $\lambda_n \geqslant 1$.
    So $M$ can be written as 
    $$M = S_{1,2,n+1} \cdot M'\,.$$
    such that $M'$ is a monomial in $S_{1,2,*}$ and $S_{1,0,*}$
    with $\#_{1,2,*}(M') = \Lambda$.
    Taking $m=n$ and $D=M'$ in the Virasoro relations \eqref{equation virasoro relations}, we have
    $$n! \,\rk(E)\int S_{1,2,n+1} \cdot M' = \int \left( 
    R_n(M') + (1-g)\sum_{i+j=n}i!j!\ S_{1,0,i}S_{1,0,j}M' - n!\chi(E)\ S_{1,0,n} M' \right)
    $$
    Clearly, the right-hand side contains monomials with $\#_{1,2,*}$ no more than $\Lambda$.
    By the induction hypothesis, the proposition 
    holds for the polynomial in the right-hand side.
    Hence, the proposition is also true for $M$.
\end{proof}
\begin{corollary}\label{corollary intersection Sjkl from S10 and Sj1}
    Let $P$ be any polynomial in $S_{j,k,l}$-classes of weighted degree $2\dim Q$.
    Then there exists a polynomial $\wt P$ only in $S_{1,0,l}$ and $S_{j,1,l}$-classes such that 
    $$ \int P = \int \wt P (S_{1,0,*},S_{j,1,*})\,.$$
\end{corollary}
\begin{proof}
    Writing $P$ as
    $$P = \sum_{M} P_{M}(S_{1,0,*},S_{1,2,*}) \cdot M(S_{j,1,*})\,,$$
    where $M$ varies over all the monomials in $S_{j,1,l}$-classes,
    with the degree constraint.
    Now the corollary follows directly by applying Proposition \ref{proposition intersection of S_12 classes from S_10 classes}
    on $P_M$'s.
\end{proof}
It is more common to consider the Kunneth decomposition of the Chern classes of the universal bundle $\K^\vee$, i.e.,
$$ c_i(\K^\vee) = a_i \otimes 1 + \sum_{j=1}^{2g} b_i^j \otimes \e_{j,1} + f_i \otimes \eta \,.
$$
Here $a_i \in H^{2i}(\quot_{r,d}(E))$, $b_i^j \in H^{2i-1}(\quot_{r,d}(E))$ and 
$f_i \in H^{2i-2}(\quot_{r,d}(E))$.
The classes $a_i,b_i^j$ and $f_i$'s can be expressed in terms of $S_{j,k,l}$'s using Newton's identities.
Precisely, let $\mu_l$'s denote the polynomials appearing in Newton's identities to express power sums in terms of elementary symmetric polynomials.
For example, 
$$\mu_2(X_1,X_2)=X_1^2 - 2X_2 \quad \text{ and }
\quad \mu_3(X_1,X_2,X_3)= X_1^3 - 3X_1X_2 + 3X_3\,.$$
Then we have 
$$ \ch_l(\K^\vee) = \frac{1}{l!} \, \mu_l(c_1(\K^\vee, \dots, c_l(\K^\vee)))\,. $$
Let us write $\mu_l(a_1,\dots,a_l)$ as $\mu_l(a)$.
Comparing the Kunneth components on both sides, one can deduce that
\begin{align}\label{equation Sjkl to abf}
\nonumber
S_{1,0,l} &= \frac{1}{l!} \, \mu_l(a) \,,\\
S_{j,1,l} &= \frac{1}{l!} \sum_{i=1}^l \frac{\p \mu_l(a)}{\p a_i} \, b_i^j \,, \\
\nonumber
S_{1,2,l} &= \frac{1}{l!} \sum_{i=1}^l \frac{\p \mu_l(a)}{\p a_i} \, f_i 
- \frac{1}{2 \cdot l!} \sum_{i,k=1}^l \frac{\p^2 \mu_l(a)}{\p a_i \p a_k} \, \Gamma_{ik}
\end{align}
where
\begin{align}\label{equation Gamma_ik}
 \Gamma_{ik} = \sum_{j=1}^g (b_i^j b_k^{j+g} + b_k^j b_i^{j+g})\,.
 \end{align}
One can also write the Chern classes in terms of the Chern characters. For example, we have the following relations for $l=1,2$.
\begin{align*}
a_1=S_{1,0,1}\,, \qquad  b_1^j=S_{j,1,1} \,, 
\qquad
f_1 = S_{1,2,1}\,, \qquad a_2= \frac{1}{2}S_{1,0,1}^2 - S_{1,0,2} \,,
\end{align*}
\begin{align*}
b_2^j = S_{1,0,1}S_{j,1,1}- S_{j,1,2}\,,  \qquad 
f_2=S_{1,0,1}\, S_{1,2,1}-\sum_{i=1}^g S_{i,1,1}\, S_{g+i,1,1}-S_{1,2,2}.
\end{align*}
Corollary \ref{corollary intersection Sjkl from S10 and Sj1} gives the following result for $a,b$ and $f$-intersections.
\begin{theorem}\label{theorem intersection fl times P}
    For any polynomial $P(a,b,f)$ of weighted degree $2\dim Q$, 
    there exists a polynomial $\wt P(a,b)$
    such that
    $$ \int P(a,b,f) = \int \wt P(a,b)\,.$$
    Moreover, $\wt P$ can be computed using the recursive formula given below.
    Let $G(a,b,f)$ be any polynomial 
    of weighted degree $2(\dim Q - l)$.
    Then 
    \begin{align}\label{equation recursive f_lG}
    \nonumber
        (-1)^{l} \cdot \rk(E) \int f_{l+1} \cdot G = &\int R_l(G) + 
        (1-g) \sum_{i+j=l} \int \mu_i(a) \mu_j(a) \, G \ - \chi(E) \cdot \int \mu_l(a)\, G \\
        & - \frac{\rk(E)}{l+1} \cdot \int \left(
        \sum_{i=1}^{l} \frac{\p \mu_l(a)}{\p a_i} \, f_i 
        - \frac{1}{2} \sum_{i,k=1}^{l+1} \frac{\p^2 \mu_l(a)}{\p a_i \p a_k} \, \Gamma_{ik}
        \right) \cdot G \,.
    \end{align}
    where by ${R_l(G)}$, we mean the polynomial in $a,b,f$-classes
    obtained by expressing 
    $G$ in terms of $S_{j,k,l}$-classes first and then
    applying the operator $R_l$.
\end{theorem}
\begin{proof}
    The first part, i.e., the existence of $\wt P(a,b)$ 
    follows directly from
    Corollary \ref{corollary intersection Sjkl from S10 and Sj1},
    using the relations \eqref{equation Sjkl to abf}.

    Taking $D=G(a,b,f)$ and $m=l$ in the Virasoro relations \eqref{equation virasoro relations} and then using the 
    relations \eqref{equation Sjkl to abf}, 
    we get the equation \eqref{equation recursive f_lG}.

    Now let $P(a,b,f)$ be any polynomial of weighted degree $\dim Q$.
    To compute $\wt P$, it is enough to assume $P$ to be a monomial.
    Let $l+1$ be the maximum of all indices $i$ for which $f_i$ appears in $P$.
    Then $P = f_{l+1} \cdot P'$ for some polynomial $P'$ of weighted degree $2(\dim Q - l)$.
    Using equation \eqref{equation recursive f_lG}, we have
    $$ \int P(a,b,f) = \int f_{l+1} \cdot P' = \int H(a,b,f)$$
    for some polynomial $H$
    such that the $f$-degree of $H$ is less than that of $P$.
    Now we do the same for $H$ and so on till we reach a polynomial with $f$-degree 0, which will be $\wt P$.
    For example, if $P$ is linear in $f$-classes, i.e.,
    $P = f_{l+1} \cdot P'(a,b)$, then $H$ is our $\wt P$. 
\end{proof}

Theorem \ref{theorem intersection fl times P} gives an algorithm to evaluate any polynomial 
$P(a,b,f)$ in terms of integrals of $a$ and $b$-classes. 
Below, we give some examples of evaluating $f$-classes in terms of intersection numbers of $a$ and $b$-classes
when $\dim Q$ is small.
\begin{example}
    Let $\dim Q =1$.
    Then evaluating the class $f_2$ against the fundamental class makes sense.
    Taking $D=1$ and $m=1$ in \eqref{equation virasoro relations},
    we get the following equation,
    $$ N \cdot \int S_{1,2,2} = 
    (2(N-r)(1-g) - \chi(E)) \cdot \int S_{1,0,1}\,.$$
    Let us write $\sum_{j=1}^g b_1^j b_1^{j+g}$ as $\theta$.
    Putting $S_{1,0,1}=a_1$ and 
    $S_{1,2,2}=(d-\deg(E))a_1- \theta - f_2$,
    we get the following
    $$\int f_2 = \left(\frac{1-(N-r-1)\deg(E) + (r^2+2r-N-Nr)(1-g)}{N} \right) \cdot \int a_1 - \int \theta\,.$$
    Equivalently,
    $$\int f_2 = (1-(r^\vee-1)\chi(E) + r^\vee(r^\vee-2)(1-g)) \cdot \frac{1}{N} \cdot \int a_1 - \int \theta\,,$$
    where $r^\vee := N-r$ is the rank of the bundle $\K^\vee$.
\end{example}

    \begin{example}\label{example integral in dim 2}
    Let $\dim Q =2$.
    We evaluate the class $f_3$.
    Taking $D=1$ and $m=2$ in \eqref{equation virasoro relations},
    we get the following equation,
    \begin{equation*}
        2N \int S_{1,2,3} 
        = \int (1-g)[4(N-r)S_{1,0,2}+S_{1,0,1}^2] - 2\nu S_{1,0,2}\,.
    \end{equation*}
    We can evaluate $f_3$ using the following identity
    $$f_3=2S_{1,2,3} - a_1^2f_1 +a_1f_2 + a_2f_1-2a_1 \theta + \Gamma_{12}\,,$$
    where $\theta = \sum_{j=1}^g b_1^j b_1^{j+g}$ and
    $\Gamma_{12} = \sum_{j=1}^g (b_1^j b_2^{j+g} + b_2^j b_1^{j+g})$.
    However, we still need to evaluate the term $a_1f_2$. 
    For that, we again use \eqref{equation virasoro relations} with $D=S_{1,0,1}=a_1$ and $m=1$,
    which gives
    $$ N \int S_{1,2,2} S_{1,0,1} 
    = \int 2 S_{1,0,2} + 2(N-r)(1-g)S_{1,0,1}^2 - \nu S_{1,0,1}^2\,.$$
    Simplifying and writing $r^\vee := N-r$, the rank of the bundle $\K^\vee$, we get
    \begin{equation*}
        \int a_1 f_2 = \int \left( \frac{1-(r^\vee-1)\chi(E)+r^\vee(r^\vee-2)(1-g)}{N} \right) a_1^2 + \frac{2}{N} \ a_2 -  a_1 \theta 
    \end{equation*}
    and
    \begin{align*}
        \int f_3 = \int \left( \frac{4 - (r^\vee-2)\chi(E) + r^\vee(r^\vee-4)(1-g)}{N} \right) a_2 - \frac{g}{N} \ a_1^2 + a_1 \theta + \Gamma_{12}\,.
    \end{align*}
    \end{example}

\vspace{.cm}

\begin{remark}\label{remark intersection theory on pairs}
    Given rank $r$ and degree $d$, let $\alpha=(\rk(E)-r,d-\deg(E)) \in K(C)$.
    We have seen in Remark \ref{remark Virasoro for pairs} that
    the Virasoro constraints are also satisfied for the moduli space $P(E,\alpha)$.
    So a similar calculation as in this section can also be carried out for $P(E,\alpha)$.
    When $E= \O^{\oplus N}$, this space parametrizes bundles
    of type $\alpha$ on $C$ with $N$ sections. 
    Intersection theory on this space is studied in \cite{Mar07}
    and it is shown that for certain polynomials, say $P$,
    the intersection number of $P$ on 
    the Quot scheme $\quot_{r,d}(\O^{\oplus N})$
    and the moduli space $P(\O^{\oplus N},\alpha)$ are equal.
    The results in this section show that if all $a$ and $b$-intersections are equal on both of these spaces, 
    then the same is true for all the tautological integrals.
    However, this equality in general is likely to be false
    due to the dependence of $f$-intersections on the boundary loci.
    See \cite{Mar07} for more details.
\end{remark}

\vspace{.2cm}
Once we have enough information about the invariant 
$[\quot_{r,d}(E)]_\inv$, we can compute all the tautological integrals using the invariant.
Recall that when $r=\rk(E)-1$, 
we computed the invariant $[\quot_{\rk(E)-1,d}(E)]_\inv$ 
in Section \ref{section computing invariant}.
Using this invariant,
we can evaluate all the tautological classes against
$[\quot_{\rk(E)-1,d}(E)]^\vir$.
In this case, the bundle $\K^\vee$ has rank 1, so we only have
$a_1, b_1^j$ classes.
Moreover, it is known that the product $\prod_{i=1}^s b_1^{j_i}b_1^{j_i+g}$
evaluates to zero unless $s \leqslant g$ and the $j_i$'s do not repeat.
\begin{theorem}\label{theorem intersections on Quot N-1}
    Let $\rk(E)=N$ and $\chi(E) = \nu$.
    Let $0 \leqslant s \leqslant g$ and $1 \leqslant j_1 < \cdots < j_s \leqslant g$ be integers
    such that the class $\prod_{i=1}^s b_1^{j_i}b_1^{j_i+g}$
    is nonzero.
    Then
    \begin{equation*}
        \int\limits_{[\quot_{N-1,d}(E))]^\vir} 
        \prod_{i=1}^s \left(b_1^{j_i}b_1^{j_i+g}\right) \cdot 
        a_1^{(\dim Q-s)} = N^{g-s}\,.
    \end{equation*}
    where $\dim Q = \nu + N(d - \deg(E)) - (1-g)$ is the expected dimension of $\quot_{N-1,d}(E)$.
\end{theorem}
\begin{proof}
We recall the setup.
Let $\du = d -\deg(E)$.
We have the translation operator
$$\acute T : H_{\bullet}(\Mdash_{(1,\du),1}) \to  H_{\bullet+2}(\Mdash_{(1,\du),1})\,.$$
Then Theorem \ref{theorem invariant quot N-1} says that the invariant 
    $[\quot_{N-1,d}(E))]_\inv$ is given by
    \begin{equation}\label{equation quot N-1 invariant}
    [\quot_{N-1,d}(E))]_\inv = 
    \res_{z=0} \left( \frac{1}{z^{\nu+N\du}} \cdot \rho(z) \cdot \sigma\left(\frac{N}{z} -s_{1,2,2}\right)\right) + \textnormal{im} \acute{T}\,,
    \end{equation}
    in $H_\bullet(\Mdash_{(1,\du),1})/\im \acute T$,
    where $\rho$ and $\sigma$ are defined as in \eqref{equation definition rho}.
Let 
$$\acute{T}^* : H^\bullet(\Mdash_{(1,\du),1}) \to H^{\bullet -2}(\Mdash_{(1,\du),1})$$ 
denote the dual of the map $\acute{T}$. 
The cohomology ring $H^\bullet(\Mdash^\pl_{(1,\du),1})$, 
being the dual of $H_\bullet(\Mdash^\pl_{(1,\du),1})$,
can be identified as $\ker (\acute{T}^*)$ inside $H^\bullet(\Mdash_{(1,\du),1})$.
So the induced pairing between $H_\bullet(\Mdash_{(1,\du),1})/\im \acute{T}$ and $\ker (\acute{T}^*)$
gives the pairing between
$H_\bullet(\Mdash^\pl_{(1,\du),1})$ and $H^\bullet(\Mdash^\pl_{(1,\du),1})$.
The map $\acute{T}^*$ is a derivation with the action on $S_{j,k,l}$ is given by
$$\acute{T}^* (S_{j,k,l}) = S_{j,k,l-1} \,.$$
So in our case,  
\begin{equation}\label{elements of kernel Tdual}
(S_{1,2,2} - dS_{101}),\, (S_{j,1,1}S_{j+g,1,1}), \, 
(S_{+,0,1} - S_{1,0,1})
\in \ker(\acute{T}^*)\,.
\end{equation}
The classes $S_{j,k,l}$'s and $a_i, b_j^j, f_i$'s have the following relations,
\begin{equation}\label{relations of tautological classes}
    \mf h^* S_{1,0,1}=a_1, \quad  
    \mf h^* S_{j,1,1}=b_1^j, 
\quad \text{ and }  
\quad \mf h^* S_{1,2,2} = d^\vee a_1 - \sum_{j=1}^g {b_1^jb_1^{j+g}} \,.
\end{equation}
Also, note that the pullback of the universal sheaf $\V_1$ on the Quot scheme is trivial, which gives us 
$$ \mf h^*(S_{+,0,l}) = 0 \qquad \text{ for any } l>0 \,.$$
    Given $s$ and $j_1,\dots,j_s$ as in the statement of theorem,
    let us consider the class
    $$\zeta:= \prod_{i=1}^s \left(S_{j_i,1,1}S_{j_i+g,1,1} \right) \cdot
    (S_{1,0,1}-S_{+,0,1})^{\dim Q-s} \quad \in 
    H^\bullet(\Mdash_{(1,\du),1})\,.$$
    Using \eqref{elements of kernel Tdual}, we see that the class
    $\zeta$ is in $\ker(\acute{T}^*)$
    and using \eqref{relations of tautological classes},
    we have 
    $$\mf h^*(\zeta) = \prod_{i=1}^s \left(b_1^{j_i}b_1^{j_i+g}\right) \cdot 
        (a_1)^{\dim Q-s}\,.$$
    So the pairing of the cycle $[\quot_{N-1,d}(E))]_\inv$ and the class $\zeta$ gives us the required intersection, i.e.
    \begin{align*}
        \int\limits_{[\quot_{N-1,d}(E))]^\vir} 
        \prod_{i=1}^s  \left(b_1^{j_i}  b_1^{j_i+g}\right) \cdot (a_1)^{\dim Q-s} \
        = \left\langle  \zeta \,, \ [\quot_{N-1,d}(E))]_\inv \right\rangle\,.
    \end{align*} 
We simplify the residue appearing in the expression for 
$[\quot_{N-1,d}(E))]_\inv$ in
\eqref{equation quot N-1 invariant}. 
Let $\sigma_i$ be the coefficient of $(N/z)^i$ 
in the expression $\sigma(N/z-s_{1,2,2})$.
Clearly $\sigma_i$ is the
$i$-th symmetric polynomial in the variables
$$(s_{1,1,1}s_{1+g,1,1}-s_{1,2,2}), \cdots, (s_{g,1,1}s_{2g,1,1}-s_{1,2,2})\,.$$
Then the residue is equal to
\begin{align*}
    \sum_{k=0}^g \left( N^k \cdot \sigma_{g-k} \cdot \{z^{\nu+N(d-\deg(E))-1+k}\}(\rho(z)) \right)
\end{align*}
where $\{z^l\}(\rho(z))$ denotes the coefficient of $z^l$ in $\rho(z)$.
We write this expression as,
\begin{align*}
    \sum_{k=0}^g \left( N^k \cdot \sigma_{g-k} \cdot \{z^{\dim Q-g+k)}\}(\rho(z)) \right)
\end{align*}
So 
\begin{align*}
        \int\limits_{[\quot_{N-1,d}(E))]^\vir} 
        \prod_{i=1}^s & \left(b_1^{j_i}  b_1^{j_i+g}\right) \cdot (a_1)^{\dim Q-s}
         = \left\langle  \zeta \,, \ \sum_{k=0}^g \left( N^k \cdot \sigma_{g-k} \cdot \{z^{\dim Q-g+k)}\}(\rho(z)) \right) \right\rangle \,.
    \end{align*} 
This can be easily calculated to be $N^{g-s}$.
\end{proof}

\begin{remark}
    The product $\prod_{i=1}^s b_1^{j_i}b_1^{j_i+g}$ may be zero even when
    $0 < s \leqslant g$ and $1 \leqslant j_1 < \cdots < j_s \leqslant g$.
    The reason is the following.
    In this case, i.e., when $r=N-1$,
    we have a map $$\det: \quot_{N-1,d}(E) \to \Pic^{d^\vee}(C)$$
    which sends a quotient to the dual of its kernel.
    Let $\L$ denote the Poincaré line bundle on $C \times \Pic^{d^\vee}(C)$ such that $\det^*\L = \K^\vee$.
    We consider the Kunneth decomposition of the Chern classes of $\L$ to get the classes 
    $\mathfrak{b}_1^j \in H^1(\Pic^{d^\vee}(C))$. 
    The classes $b_1^j$'s
    are the pullbacks of the corresponding classes $\mathfrak{b}_1^j$'s via $\det$.
    So the product $\prod_{i=1}^s (b_1^{j_i}b_1^{j_i+g})$ becomes zero when $s> \dim (\im(\det))$.
    For example, when $E$ is the trivial bundle $\O_C^{\oplus N}$, the image of the map $\det$ is the Brill-Noether locus
    $W_0^d(C) = \{L \in \Pic^{d}(C) : H^0(L)\neq 0\}$ which has dimension $d$. So the product vanishes for $s>d$.
\end{remark}

\begin{remark}
An important special case is when $\dim Q =0$, 
so that the Quot scheme is a bunch of points.
Then the intersection of the class 1 over the virtual fundamental class $[\quot_{N-1,d}(E))]^\vir$ is $N^g$,
which gives the virtual count of maximal rank-one subbundles of $E$.
It is known that, for a general stable bundle $E$
and an integer $d$ such that $\dim \quot_{r,d}(E)=0$, 
the Quot scheme $\quot_{r,d}(E)$ is smooth.
Hence, for a general stable bundle $E$, the virtual count $N^g$ gives the number of maximal rank-one
subsheaves of $E$.
This result is already well known in greater generality; see, for example, \cite[Corollary 4.3]{Hol04}.
\end{remark}

    \vspace{1cm}

\newcommand{\etalchar}[1]{$^{#1}$}

\end{document}